\documentclass{amsart}
\usepackage{amssymb,amstext,amsmath,amscd,amsthm,amsfonts,enumerate,graphicx,latexsym,color}
\usepackage[all]{xy}
\newtheorem{thm}{Theorem}[section]
\newtheorem{cor}[thm]{Corollary}
\newtheorem{lem}[thm]{Lemma}
\newtheorem{prop}[thm]{Proposition}

\newtheorem*{thm*}{Theorem}
\theoremstyle{definition}
\newtheorem{conv}[thm]{Convention}
\newtheorem{defn}[thm]{Definition}
\newtheorem{rem}[thm]{Remark}
\newtheorem{ques}[thm]{Question}

\newtheorem*{conj*}{Conjecture}
\newtheorem{exam}[thm]{Example}

\newtheorem*{claim*}{Claim}
\newtheorem*{ques*}{Question}

\theoremstyle{remark}
\newtheorem*{ac}{Acknowledgments}
\numberwithin{equation}{thm}
\def\A{\mathcal{A}}

\def\ann{\operatorname{Ann}}
\def\b{\operatorname{B}}
\def\Cdim{\operatorname{CI-dim}}
\def\cext{\mathrm{\widehat{Ext}}\mathrm{}}

\def\codepth{\operatorname{codepth}}
\def\Coker{\operatorname{Coker}}
\def\Cone{\operatorname{Cone}}
\def\ctor{\mathrm{\widehat{Tor}}\mathrm{}}
\def\cx{\operatorname{cx}}

\def\depth{\operatorname{depth}}
\def\E{\operatorname{E}}

\def\edim{\operatorname{edim}}
\def\Ext{\operatorname{Ext}}

\def\finpql{\operatorname{finqpl}}
\def\ge{\geqslant}
\def\Gdim{\operatorname{G-dim}}
\def\grade{\operatorname{grade}}
\def\H{\operatorname{H}}
\def\height{\operatorname{ht}}
\def\hinf{\operatorname{hinf}}
\def\Hom{\operatorname{Hom}}
\def\hsup{\operatorname{hsup}}

\def\Im{\operatorname{Im}}
\def\K{\operatorname{K}}
\def\Ker{\operatorname{Ker}}

\def\le{\leqslant}
\def\len{\operatorname{length}}
\def\m{\mathfrak{m}}
\def\Min{\operatorname{Min}}
\def\Mod{\operatorname{Mod}}
\def\mod{\operatorname{mod}}

\def\N{\mathbb{N}}
\def\p{\mathfrak{p}}
\def\pd{\operatorname{pd}}

\def\qpd{\operatorname{qpd}}
\def\qpl{\operatorname{qpl}}
\def\r{\operatorname{r}}

\def\soc{\operatorname{Soc}}

\def\syz{\mathrm{\Omega}}
\def\Tor{\operatorname{Tor}}
\def\tr{\operatorname{Tr}}
\def\xx{\boldsymbol{x}}
\def\Z{\mathbb{Z}}
\def\z{\operatorname{Z}}
\begin{document}
\title{Quasi-projective dimension}
\author{Mohsen Gheibi}
\address[M. Gheibi]{Department of Mathematics, University of Texas at Arlington, 411 S. Nedderman Drive, Pickard Hall 445, Arlington, TX 76019, USA }
\email{mohsen.gheibi@uta.edu}
\urladdr{https://mentis.uta.edu/explore/profile/mohsen-gheibi}
\author{David A. Jorgensen}
\address[D. A. Jorgensen]{Department of Mathematics, University of Texas at Arlington, 411 S. Nedderman Drive, Pickard Hall 429, Arlington, TX 76019, USA }
\email{djorgens@uta.edu}
\urladdr{http://www.uta.edu/faculty/djorgens/}
\author{Ryo Takahashi}
\address[R. Takahashi]{Graduate School of Mathematics, Nagoya University, Furocho, Chikusaku, Nagoya 464-8602, Japan}
\email{takahashi@math.nagoya-u.ac.jp}
\urladdr{http://www.math.nagoya-u.ac.jp/~takahashi/}
\thanks{Ryo Takahashi was partly supported by JSPS Grant-in-Aid for Scientific Research 16K05098, 19K03443 and JSPS Fund for the Promotion of Joint International Research 16KK0099}
\dedicatory{Dedicated to Professors Roger and Sylvia Wiegand on the occasion of their 150th birthday}
\begin{abstract}
In this paper, we introduce a new homological invariant called quasi-projective dimension, which is a generalization of projective dimension.
We discuss various properties of quasi-projective dimension.
Among other things, we prove the following.
(1) Over a quotient of a regular local ring by a regular sequence, every finitely generated module has finite quasi-projective dimension.
(2) The Auslander--Buchsbaum formula and the depth formula for modules of finite projective dimension remain valid for modules of finite quasi-projective dimension.
(3) Several results on vanishing of Tor and Ext hold for modules of finite quasi-projective dimension.
\end{abstract}
\maketitle
\tableofcontents
\section{Introduction}

Homological invariants such as projective dimension, injective dimension, and flat dimension are central themes in ring theory and homological algebra. They not only give information about modules, but they also provide a means for classifying associative rings. For example, a commutative Noetherian local ring $R$ is regular if and only if every finitely generated $R$-module $M$ has finite projective dimension, that is, $M$ can be approximated in finitely many steps by projective modules.

To study modules over non-regular rings, several homological invariants that generalize projective dimension
($\pd$) have been defined: Auslander and Bridger \cite{AB} defined and studied Gorenstein dimension ($\Gdim$). Next, Avramov \cite{Av2} defined and studied virtual projective dimension, and later,  Avramov, Gasharov and Peeva \cite{AGP} developed it to complete intersection dimension ($\Cdim$). The main point of these homological invariants  is that one extends the class of modules used to resolve; each dimension is equal to the minimal length of resolutions by modules the dimensions of which are zero.

The goal of this paper is to introduce a new homological invariant which also generalizes projective dimension, and study its properties. However, unlike the invariants $\Gdim$ and $\Cdim$, our invariant is based on extending the notion of what is a resolution. We say that a module $M$ over an associative ring $R$ has {\em finite quasi-projective dimension} if there exists a finite complex of projective modules whose homologies are isomorphic to finite direct sums of copies of $M$.  We use such complexes to define the new invariant {\em quasi-projective dimension}, $\qpd_RM$, for $R$-modules $M$ (see Definition \ref{12}). A remarkable distinction between quasi-projective dimension and the invariants $\pd$, $\Gdim$ and $\Cdim$ is that the residue field of a commutative local ring always has finite quasi-projective dimension. The finiteness of these other invariants of the residue field, in contrast, force the ring to be regular, Gorenstein, and a complete intersection, respectively.

Just as for the invariants $\pd$, $\Gdim$ and $\Cdim$, our invariant satisfies a version of the celebrated {\em Auslander--Buchsbaum formula}.

\begin{thm} For a finitely generated module $M$ over a commutative Noetherian local ring $R$, if
$\qpd_RM<\infty$, then $$\qpd_RM=\depth R-\depth_RM.$$
\end{thm}

Thus all invariants, $\pd$, $\Gdim$, $\Cdim$, and $\qpd$ agree whenever finite.
We prove Theorem 1.1 in Section 4, as Theorem \ref{AB}.

A common theme of the results of this paper is that modules over a commutative Noetherian local ring with finite quasi-projective dimension behave homologically like modules over a complete intersection, or, more generally, modules of finite complete intersection dimension.  This is not surprising in view of the fact that if $R$ is the quotient of a regular local ring by a regular sequence, then every $R$-module has finite quasi-projective dimension (see Corollary \ref{36} below). For example, modules of finite quasi-projective dimension also satisfy {\em Auslander's depth formula} \cite[Theorem 1.2]{A}.

\begin{thm} Suppose $R$ is a commutative Noetherian local ring, $M$ and $N$ are finitely generated $R$-modules such that $\qpd_RM<\infty$ and $\Tor^R_i(M,N)=0$ for all $i\geq 1$. Then $$\depth_RM+\depth_RN=\depth R+\depth_RM\otimes_RN$$
\end{thm}

Huneke and Wiegand \cite{HW} established the depth formula for Tor-independent modules over complete intersection local rings. Later, Araya and Yoshino \cite{AY} showed that the depth formula holds for finitely generated Tor-independent modules $M$ and $N$ over an arbitrary local ring provided one of $M$ or $N$ has finite complete intersection dimension. We prove Theorem 1.2 in Section 4 as Theorem \ref{depthformula}.

Symmetry in vanishing of Ext first was proven by Avramov and Buchweitz \cite{AB2} for finitely generated modules over complete intersections. Their proof relied on their development of support varieties for pairs of finitely generated modules. Later, Huneke and Jorgensen \cite{HJ} introduced a class of Gorenstein local rings, namely, {\em AB rings}, and generalized symmetry in vanishing of Ext over AB rings. Every local complete intersection ring is an AB ring, but not conversely, and there are examples of Gorenstein rings which are not AB rings. In Section 6, we prove the following symmetry in vanishing of Ext result over an arbitrary Gorenstein ring (see Theorem \ref{24}).

\begin{thm} Let $R$ be a Gorenstein ring, and $M$ and $N$ be finitely generated $R$-modules. Assume that
$\qpd_RM<\infty$. Then
$$
\Ext^i_R(M,N)=0 \text{ for all $i\gg 0$ if and only if }  \Ext^i_R(N,M)=0 \text{ for all $i\gg 0$}
$$
\end{thm}

We also prove in Section 6 several rigidity of Tor and Ext results for modules of finite quasi-projective dimension, which are akin to rigidity results for modules over a complete intersection.
Another consequence of the theory of support varieties over a complete intersection $R$ is that the long-standing
{\em Auslander--Reiten Conjecture} holds for $R$: {\em for every finitely generated $R$-module $M$ we have
$\Ext^i_R(M,M)=0$ for all $i\gg 0$ if and only if $\pd_RM<\infty$}. This result was originally proven by Auslander, Ding and Solberg using different techniques; see \cite{ADS}. The Auslander--Reiten conjecture also holds over AB rings and some classes of Artinian local rings; see \cite{CH} and \cite{HSV}, but even over Gorenstein rings, it is open. We show in Theorem \ref{26} in Section 6 that modules of finite quasi-projective dimension satisfy the condition of the Auslander--Reiten conjecture.

\begin{thm} Let  $M$ be a finitely generated $R$-module such that $\qpd_RM<\infty$. Then $\Ext^i_R(M,M)=0$ for all $i> 0$ if and only if $M$ is projective.
\end{thm}

In Section 4 we investigate whether over a fixed ring there is a bound on the lengths of the minimal finite complexes of projectives realizing the modules of finite quasi-projective dimension.

In Section 7, we focus on ideals $I$ of a commutative Noetherian ring $R$ whose Koszul homologies are free
$R/I$-modules. These ideals have finite quasi-projective dimension. A maximal ideal of $R$ is such an example, however, there are many other examples.  We call these ideals {\em FKH ideals}. It turns out that if $I$ is an FKH ideal then $\qpd_RR/I=\grade_RI$.  Among other results, we prove that over a Gorenstein ring $R$, if $J$ is linked to an FKH ideal $I$, then $\qpd_RR/J = \grade_RJ$. This result is an analogue of a theorem of Peskine and Szpiro \cite{PS} for linkage of perfect ideals over Gorenstein rings.

\begin{ac}
The authors thank  Srikanth Iyengar for pointing out \cite{BI} and Josh Pollitz for very useful discussions.
Part of this work has been done during the visit of Ryo Takahashi to the University of Texas at Arlington in June, 2019. He is grateful for their kind hospitality. Finally, the authors thank the referee for his/her valuable suggestions and comments that improved this paper. 
\end{ac}

\section{Preliminaries}

\begin{conv}\label{co}
Throughout this paper, unless otherwise stated, all rings are commutative Noetherian rings with identity.
Let $R$ be a ring.
We denote by $\Mod R$ the category of $R$-modules, and by $\mod R$ the full subcategory consisting of finitely generated $R$-modules.
The set of nonnegative integers is denoted by $\N$.
When $R$ is local, $\edim R$ stands for the embedding dimension of $R$.
We put $(-)^*=\Hom_R(-,R)$.
For $n\in\N\cup\{\pm\infty\}$ we say that $n$ is finite and write $n<\infty$ if $n\in\N\cup\{-\infty\}$.
\end{conv}

\begin{defn}
Let $\A$ be an abelian category.
Let $X= (\cdots\xrightarrow{\partial_{i+2}}X_{i+1}\xrightarrow{\partial_{i+1}}X_i \xrightarrow{\partial_{i}} X_{i-1} \xrightarrow{\partial_{i-1}} \cdots)$ be a complex of objects of $\A$.
For each integer $i$, we define the $i$th cycle $\z_i(X)=\Ker\partial_i$, the $i$th boundary $\b_i(X)=\Im\partial_{i+1}$ and the $i$th homology $\H_i(X)=\z_i(X)/\b_i(X)$.
The {\em supremum}, {\em infimum}, {\em homological supremum} and {\em homological infimum} of $X$ are defined by
$$
\begin{cases}
\sup X=\sup\{i\in\Z\mid X_i\ne0\},\\
\inf X=\inf\{i\in\Z\mid X_i\ne0\},
\end{cases}
\qquad
\begin{cases}
\hsup X=\sup\{i\in\Z\mid\H_i(X)\ne0\},\\
\hinf X=\inf\{i\in\Z\mid\H_i(X)\ne0\}.
\end{cases}
$$
We define the \emph{length} of $X$ to be $\len X=\sup X-\inf X$.
We say that $X$ has \emph{finite length}, or is \emph{bounded}, if $\len X<\infty$. We say that $X$ is \emph{bounded below} if $\inf X>-\infty$. Note that if $X$ satisfies $X_i=0$ for all 
$i\in\mathbb Z$, then $\sup X =-\infty$, $\inf X=\infty$, and by convention we set $\len X=-\infty$.

For an integer $j$, the complex $X[j]$ is defined by ${X[j]}_i = X_{i-j}$ and $\partial^{X[j]}_i=(-1)^j\partial^X_{i-j}$ for all $i$.
\end{defn}

\begin{defn}
Let $(R,\m,k)$ be a local ring. A complex $(F,\partial)$ of free $R$-modules of finite rank is called {\em minimal} if $\partial_i\otimes_Rk=0$ for all $i$.
\end{defn}

We recall the definition of projective dimension for objects in an abelian category.

\begin{defn}
\begin{enumerate}[(1)]
\item
Let $\A$ be an abelian category with enough projective objects.
Let $M\in\A$.
\begin{enumerate}[(a)]
\item
A complex $P=(\cdots\to P_2\to P_1\to P_0\to0)$ of projective objects of $\A$ is called a {\em projective resolution} of $M$ if $\H_i(P)=0$ for all $i>0$ and $\H_0(P)=M$.
\item
For each integer $i>0$ we define an $i$th {\em syzygy} of $M$ by $\syz_\A^iM=\b_i(P)$, where $P$ is a projective resolution of $M$.
We simply write $\syz_\A M=\syz_\A^1M$.
We put $\syz_\A^0M=M$ and call it a $0$th syzygy of $M$.
For each $i\ge0$ the object $\syz_\A^iM$ is uniquely determined up to projective summands.
\item
The {\em projective dimension} $\pd_\A M$ of $M$ is by definition the infimum of integers $n\ge0$ such that there exists a projective resolution $P$ of $M$ with $P_i=0$ for all $i>n$.
\end{enumerate}
\item
For an $R$-module $M$, a projective resolution of $M$, syzygies $\syz_R^iM$ of $M$, and the projective dimension $\pd_RM$ of $M$ are defined by setting $\A:=\Mod R$ in (1).
Suppose that $R$ is local and $M$ is finitely generated.
Then one can take a minimal free resolution $F$ of $M$, which is uniquely determined up to isomorphism of complexes.
Whenever we work in this setting, we define the syzygies of $M$ by using $F$, so that they are uniquely determined up to isomorphism.
\end{enumerate}
\end{defn}

We close this section with the following lemma which we will use frequently in the rest of the paper.

\begin{lem}\label{spseq}
Let $P$ be a bounded below complex of projective $R$-modules.
Then for an $R$-module $N$ there are convergent spectral sequences
\begin{align}
\label{sptor1}
\E^2_{p,q}&=\Tor^R_p(\H_q(P),N)\Longrightarrow \H_{p+q}(P\otimes_R N), \ \text{and}\\
\label{spext1}
\E^{p,q}_2&= \Ext^p_R(\H_q(P),N) \Longrightarrow \H^{p+q}(\Hom_R(P,N)).
\end{align}
\end{lem}

\begin{proof}
Since $P$ is a complex of projective $R$-modules, the first spectral sequence is derived from the double complex $P\otimes_RF$ where $F$ is a projective resolution of $N$, and the second one is derived from the double complex $\Hom_R(P,I)$ where $I$ is an injective resolution of $N$; see for example \cite[Theorem 11.34]{R}.
\end{proof}

\section{Definition and basic properties of quasi-projective dimension}

In this section we introduce our main objects of study, namely, modules of finite quasi-projective dimension, and then give some examples and basic properties.

\begin{defn}\label{12}
Let $\A$ be an abelian category with enough projective objects.
Let $M$ be an object of $\A$.
\begin{enumerate}[(1)]
\item
A {\em quasi-projective resolution} of $M$ in $\A$ is defined as a bounded below complex $P$ of projective objects of $\A$ such that for all $i\ge\inf P$ there exist non-negative integers 
$a_i$, not all zero, such that $\H_i(P)\cong M^{\oplus a_i}$.
We say that a quasi-projective resolution $P$ is {\em finite} if $P$ has finite length, equivalently, $\sup P<\infty$.
\item
We define the {\em quasi-projective dimension} of $M$ in $\A$ by
$$
\qpd_\A M=
\begin{cases}\inf\{\sup P-\hsup P\mid\text{$P$ is a finite quasi-projective resolution of $M$}\} & (\text{if $M\ne0$}),\\
-\infty & (\text{if $M=0$}).
\end{cases}
$$
\end{enumerate}
We simply set $\qpd_RM=\qpd_{\Mod R}M$ for an $R$-module $M$.
Note that $\qpd_\A M\in\N\cup\{\pm\infty\}$. One has that $\qpd_\A M=-\infty$ if and only if $M=0$, that $\qpd_\A M=\infty$ if and only if $M$ does not admit a finite quasi-projective resolution, and that $\qpd_\A M\in\N$ if and only if $M$ is nonzero and admits a finite quasi-projective resolution.
\end{defn}

\begin{rem}\label{8}
Let $\A$ be an abelian category with enough projective objects.
Let $M$ be an object of $\A$.
\begin{enumerate}[(1)]
\item Let $P$ be a quasi-projective resolution of $M\ne 0$. Assuming that 
$\H_i(P)\cong M^{\oplus a_i}$ for all $i\in\Z$, we note that $\hsup P=\sup\{i\mid a_i\ne 0\}$
and $\hinf P=\inf\{i\mid a_i\ne 0\}$.
\item
Every (finite) projective resolution of $M$ is a (finite) quasi-projective resolution of $M$.
In particular, the zero complex is a finite quasi-projective resolution of the zero module.
\item
If $M$ has finite projective dimension, then $M$ has finite quasi-projective dimension.
More precisely, there is an inequality $\qpd_\A M\le\pd_\A M$.
\item Let $M\ne0$, and $P$ be a quasi-projective resolution of $M$.
Then there exists a quasi-projective resolution $P'$ of $M$ with $\H_i(P')=\H_i(P)$ for all 
$i\in\Z$, and $\hinf P'=\inf P'$.
Indeed, set $u=\inf P$. If $\H_u(P)=0$, then $\partial_{u+1}$ is a split surjection, and 
$\Ker\partial_{u+1}$ is projective. The truncated complex 
$\cdots\xrightarrow{\partial_{u+3}}P_{u+2}\xrightarrow{\partial_{u+2}}\Ker\partial_{u+1}\to0$ is also a quasi-projective resolution of $M$ with the same homology as $P$. Repeating this process as necessary, 
we obtain a quasi-projective resolution $P'$ of $M$ with the desired property.
\end{enumerate}
\end{rem}

Let $\A$ be an abelian category, and let $\alpha:X\to Y$ be a morphism in the category of complexes of objects of $\A$.
Then we denote by $\Cone(\alpha)$ the mapping cone of $\alpha$.

The second, third and fourth statements in the following proposition will be refined for modules over local rings; see Corollary \ref{13}.

\begin{prop}\label{Syz}
Let $\A$ be an abelian category with enough projective objects.
\begin{enumerate}[\rm(1)]
\item
For an object $M\in\A$ and an integer $n>0$, one has $\qpd_\A(M^{\oplus n})=\qpd_\A M$.
\item
Let $M,N\in\A$.
Then $\qpd_\A(M\oplus N)\le\sup\{\qpd_\A M,\qpd_\A N\}$.
\item
For a nonzero object $M\in\A$ and a projective object $J\in\A$, one has $\qpd_\A(M\oplus J)\le\qpd_\A M$.
\item
For a nonzero object $M\in\A$, there is an inequality $\qpd_\A(\syz_\A M)\le\qpd_\A M$.
\end{enumerate}
\end{prop}

\begin{proof}
(1) If $P$ is a quasi-projective resolution of $M$, then $P^{\oplus n}$ is a quasi-projective resolution of $M^{\oplus n}$.
Conversely, if $P$ is a quasi-projective resolution of $M^{\oplus n}$, then it is also a quasi-projective resolution of $M$.
The assertion now follows.

(2) We may assume $M\ne0\ne N$ and $\qpd_\A M<\infty>\qpd_\A N$.
Let $P$ and $P'$ be finite quasi-projective resolutions of $M$ and $N$, respectively.
By assumption, $\H_i(P)\cong M^{\oplus a_i}$ and $\H_j(P')\cong N^{\oplus b_j}$ for some $a_i,b_j\ge0$ (and all but finitely many of the $a_i$ and $b_j$ are zero.) Then the complex
$$
F=\big{(}\oplus_{j\in\Z}P^{\oplus b_j}[j]\big{)}\oplus \big{(}\oplus_{i\in\Z}{P'}^{\oplus a_i}[i]\big{)}
$$
is a quasi-projective resolution for $M\oplus N$; in fact, 
$\H_k(F)=(M\oplus N)^{\oplus\sum_{i+j=k}a_ib_j}$ for each $k$. Note that 
$\sup F=\max\{\sup P+\hsup P',\sup P'+\hsup P\}$ and that 
$\hsup F=\hsup P+\hsup P'$.  Therefore we have
\[
\qpd_\A(M\oplus N)\le\sup F -\hsup F=\max\{\sup P-\hsup P,\sup P'-\hsup P'\}
\]
Choosing $P$ and $P'$ such that $\qpd_\A M=\sup P-\hsup P$ and $\qpd_\A N=\sup P'-\hsup P'$
we achieve the desired conclusion.

(3) The assertion follows from (2) by letting $N=J$.

(4) Putting $N=\syz_\A M$, we have an exact sequence $0\to N\to J\xrightarrow{\pi}M\to0$ in $\A$ with $J$ projective.
What we want to show is that $\qpd_\A N\le\qpd_\A M$.
We may assume $N\ne0$, and that $\qpd_\A M<\infty$.
Let $P$ be a finite quasi-projective resolution of $M$.
Then for each $i\in\Z$ we have $\H_i(P)\cong M^{\oplus a_i}$ with $a_i\ge0$ (and all but finitely many $a_i$ are zero.) Set $G_i=J^{\oplus a_i}$ and consider the complex
\[
G=(\cdots\to G_{i+1}\xrightarrow{0}G_i\xrightarrow{0}G_{i-1}\to\cdots)
\]
As $G_i$ is projective, the map $\pi^{\oplus a_i}:G_i\to\H_i(P)$ lifts to a map $G_i\to\z_i(P)$.
Composing this with the inclusion map $\z_i(P)\to P_i$, we get a map $\alpha_i:G_i\to P_i$, and obtain a chain map $\alpha:G\to P$. The short exact sequence of complexes $0\to P \to \Cone(\alpha) \to G[1] \to 0$ yields the long exact sequence of homology 
\[
\cdots \to \H_{i+1}(P) \to \H_{i+1}(\Cone(\alpha)) \to \H_i(G) \to \H_i(P) \to \cdots.
\]
By construction of $G$, this long exact sequence breaks into short exact sequences 
\[
0\to \H_{i+1}(\Cone(\alpha))\to G_i\xrightarrow{\pi^{\oplus a_i}}\H_{i}(P) \to 0
\]
and so $\H_{i+1}(\Cone(\alpha))\cong N^{\oplus a_i}$ for all $i\in\Z$.
Thus
\[
\Cone(\alpha)=(\cdots\to G_i\oplus P_{i+1}\to G_{i-1}\oplus P_i\to\cdots)
\]
is a quasi-projective resolution of $N$, and $\hsup \Cone(\alpha)=\hsup P+1$. We also see that
$\sup \Cone(\alpha)\le \sup P+1$.  Therefore we have
\[
\qpd_\A N\le\sup\Cone(\alpha)-\hsup\Cone(\alpha)\le\sup P+1-(\hsup P+1)=\sup P-\hsup P
\]
Choosing $P$ such that $\qpd_\A P=\sup P-\hsup P$, we demonstrate the claim.
\end{proof}

Recall that a bounded complex of finitely generated projective $R$-modules is called \emph{perfect}.

\begin{prop}\label{14}
Let $M \ne 0$ be a finitely generated $R$-module with $\qpd_RM<\infty$.
Then there exists a perfect complex $P$ which is a quasi-projective resolution of $M$ 
such that $\qpd_RM=\sup P-\hsup P$.
\end{prop}

\begin{proof}
Let $G$ be a quasi-projective resolution of $M$ with $\qpd_RM=\sup G-\hsup G$.
Note that each homology of $G$ is a finitely generated $R$-module.
There is a complex $(F,\partial)$ of finitely generated projective $R$-modules with 
$\inf F=\inf G$, and a quasi-isomorphism $\alpha: F\to G$; see \cite[XVII, Propositions 1.2 and 1.3]{CE} or \cite[Theorem A.3.2]{C}.
Since $\alpha$ is a quasi-isomorphism, $\Cone(\alpha)$ is acyclic, that is, 
$\H_i(\Cone(\alpha))=0$ for all $i\in\Z$.  As $\Cone(\alpha)$ is a complex of projective modules, it is split exact (i.e., contractible), so that each $\z_i(\Cone(\alpha))$ and $\b_i(\Cone(\alpha))$
are projective.  Let $s=\sup G$.  Note that $\Cone(\alpha)_{i+1}=F_i$ for all $i\ge s$. 
By exactness of $\Cone(\alpha)$ we have that 
\[
F_{s+1}\xrightarrow{\partial_{s+1}} F_{s}\to \b_{s}(\Cone(\alpha)) \to 0
\]
is exact, and so $F_s/\b_s(F)\cong \b_s(\Cone(\alpha))$ is projective. Thus 
\[
P=(0\to F_s/\b_s(F)\xrightarrow{\overline{\partial_s}}F_{s-1}\xrightarrow{\partial_{s-1}}\cdots\xrightarrow{\partial_{\inf G+1}}F_{\inf G}\to0)
\]
is a perfect complex which is quasi-isomorphic to $F$, and so is quasi-isomorphic to $G$.
Therefore $P$ is a quasi-projective resolution of $M$, and we have $\hsup P=\hsup G$ and 
$\sup P=s$ (otherwise $F_s/\b_s(F)=0$ and $\sup P<s$, which contradicts the fact that 
$s-\hsup G=\qpd_RM$.) Therefore $\qpd M = \sup P -\hsup P$.
\end{proof}

\begin{prop}\label{7}
Let $M$ be an $R$-module.
\begin{enumerate}[\rm(1)]
\item
Let $R \to S$ be a flat ring homomorphism.
If $P$ is a quasi-projective resolution of $M$ over $R$, then $P\otimes_RS$ is a quasi-projective resolution of $M\otimes_R S$ over $S$.
Moreover, it holds that $\qpd_{S}(M\otimes_R S)\le\qpd_RM$.
\item
Let $x$ be an element of $R$ which is regular on both $R$ and $M$.
If $P$ is a quasi-projective resolution of the $R$-module $M$, then $P\otimes_RR/(x)$ is a quasi-projective resolution of the $R/(x)$-module $M/xM$.
Moreover, it holds that $\qpd_{R/(x)}M/xM\le\qpd_RM$.
\item
Let $Q$ be a local ring, $\xx=x_1,\dots,x_c$ be a $Q$-regular sequence and $R=Q/(\xx)$.
Let $M$ be a finitely generated $R$-module.
Then $\qpd_QM\le\qpd_RM+c$.
\end{enumerate}
\end{prop}

\begin{proof}
(1) We have $P\otimes_RS$ is a complex of projective $S$-modules.
By Lemma \ref{spseq}, one has $\H_i(P\otimes_R S)\cong \H_i(P)\otimes_RS$, which finishes the proof of the first assertion.

Let us show the second assertion. It clearly holds if $M\otimes_RS=0$, 
so we assume $M\otimes_RS\ne0$. In particular, $M\ne0$. We can also assume $n:=\qpd_RM<\infty$.
Then there exists a quasi-projective resolution $P$ of $M$ with $n=\sup P-\hsup P$.

Clearly $\sup(P\otimes_RS)\le\sup P$. We claim moreover that $\hsup(P\otimes_RS)=\hsup P$.
Indeed, let $h=\hsup P$.  From the isomorphisms $\H_i(P\otimes_RS)\cong\H_i(P)\otimes_RS$ we see that $\hsup(P\otimes_RS)\le h$. For each $i\in\Z$ there are integers $a_i\ge0$ such that 
$\H_i(P)\cong M^{\oplus a_i}$ and $h=\sup\{i\mid a_i\ne 0\}$. Thus we have 
$\H_h(P\otimes_RS)\cong\H_h(P)\otimes_RS\cong(M\otimes_RS)^{\oplus a_h}\ne 0$, since 
$M\otimes_RS\ne0$. Hence $\hsup(P\otimes_RS)=h$. Altogether we have
\[
\qpd_S(M\otimes_RS)\le\sup(P\otimes_RS)-\hsup(P\otimes_RS)\le\sup P-\hsup P=\qpd_RM.
\]

(2) By Lemma \ref{spseq} there is a spectral sequence
$$
\E^2_{p,q}=\Tor^R_p(\H_q(P),R/(x))\ \Longrightarrow\ \H_{p+q}(P\otimes_RR/(x)).
$$
Since $\Tor^R_{>0}(M,R/(x))=0$, the spectral sequence shows that $\H_i(P)\otimes_RR/(x)\cong \H_i(P\otimes_RR/(x))$ for each $i$. Thus the first assertion follows.

To show the second assertion, we may assume $M/xM\ne0$. Let $P$ be a quasi-projective resolution of $M$ such
that $\qpd M=\sup P-\hsup P$. Then $P\otimes_R R/(x)$ is a quasi-projective
resolution of $M/xM$. It is clear that $\sup(P\otimes_R R/(x))\le \sup P$. Put
$h=\hsup P$ and let $a>0$ be such that $\H_hP\cong M^{\oplus a}$. Then
$\H_h(P\otimes_R R/(x))\cong(\H_hP)\otimes_R R/(x)\cong(M/xM)^{\oplus a}\ne0$
since $M/xM\ne0$. Hence $\hsup(P\otimes_R R/(x))\ge h$, and therefore
$\qpd(M\otimes_R R/(x))\le \sup(P\otimes_R R/(x))-\hsup(P\otimes_R R/(x))\le \sup
P-\hsup P=\qpd M$.

(3) We may assume $M\ne 0$ and $\qpd_RM<\infty$.
Proposition \ref{14} guarantees that there exists a bounded complex $P$ of finitely generated 
free $R$-modules which is a quasi-projective resolution of $M$ over $R$ and satisfies 
$\qpd_RM=\sup P-\hsup P$. By \cite[Theorem 2.4]{BJM}, there exists a bounded complex $G$ of finitely generated free $Q$-modules with $\sup G=\sup P+c$ such that $\H_i(G)\cong\H_i(P)$ 
for all $i\in\Z$. The complex $G$ is a quasi-projective resolution of $M$ over $Q$ with 
$\hsup G=\hsup P$. Hence $\qpd_QM\le\sup G-\hsup G=(\sup P+c)-\hsup P=\qpd_RM+c$.
\end{proof}

Recall that an $R$-module $M$ is called {\em periodic} if there exist an integer $r>0$ and a projective resolution $(P,\partial)$ such that $\partial_i=\partial_{r+i}$ for all $i\ge0$.
The minimal integer $r$ satisfying this condition is called the {\em periodicity} of $M$.

The following proposition gives examples of modules of finite quasi-projective dimension which do not have finite projective dimension.

\begin{prop}\label{18}
\begin{enumerate}[\rm(1)]
\item
If $\m$ is a maximal ideal of $R$ generated by $n$ elements, then the $R$-module $R/\m$ admits a quasi-projective resolution of length $n$.
In particular, one has $\qpd_Rk\le\edim R<\infty$ for every local ring $R$ with residue field $k$.
\item
If $M\ne0$ is a periodic $R$-module of periodicity $r$, then $\qpd_RM=0$.
\end{enumerate}
\end{prop}

\begin{proof}
(1) Consider the the Koszul complex $\K(\xx)$ on a system of generators $\xx=x_1,\dots,x_n$ of  $\m$.
This is a perfect complex and each homology is a finite-dimensional vector space over $R/\m$.
Hence it is a quasi-projective resolution of the $R$-module $R/\m$.

(2) There is a projective resolution $(P,\partial)$ of $M$ with $P_i=P_{r+i}$ and $\partial_i=\partial_{r+i}$ for all $i\ge0$, where $r$ is the periodicity of $M$.
Put $s=\max\{1,r-1\}$.
Then the truncation $P'=(0\to P_s\xrightarrow{\partial_s}\cdots\xrightarrow{\partial_1}P_0\to 0)$ is a perfect complex with $\H_s(P)\cong M \cong \H_0(P)$ and $\H_i(P)=0$ for $0<i<s$.
Hence $P'$ is a quasi-projective resolution of $M$, and hence we get $\qpd_RM\le \sup P' -\hsup P' =s-s=0$.
\end{proof}

\begin{prop}\label{qperfect}
Let $Q$ be a ring, $\xx=x_1,\dots,x_c$ a $Q$-regular sequence and $R=Q/(\xx)$.
Let $M$ be an $R$-module, and $P$ be a projective resolution of $M$ as a $Q$-module.
Then $P\otimes_QR$ is a quasi-projective resolution of $M$ as an $R$-module.
In particular, one has the inequality $\qpd_RM\le\pd_QM$.
\end{prop}

\begin{proof}
Since $\xx$ is a regular sequence, the Koszul complex $\K(\xx,Q)$ is a free resolution of 
$R$ over $Q$. For each $0\le i\le c$ there are isomorphisms
$$
\H_i(P\otimes_QR) \cong \Tor^Q_i(M,R)\cong \H_i(M\otimes_Q\K(\xx,Q))\cong\H_i(\xx,M)
\cong M^{\oplus\binom{c}{i}}
$$
of $Q$-modules, where the last isomorphism holds since $\xx M=0$.
Since the map $Q\to R$ is surjective, the above isomorphisms are isomorphisms of $R$-modules.
\end{proof}

The following corollary is an immediate consequence.

\begin{cor}\label{36}
Let $R$ be a quotient of a regular local ring by a regular sequence.
Let $M$ be a finitely generated $R$-module.
Then $\qpd_RM\le\edim R$.
In particular, $M$ has finite quasi-projective dimension.
\end{cor}

By a {\em complete intersection ring}, we mean a local ring $(R,\m)$ such that its $\m$-adic completion is a quotient of a regular local ring by a  regular sequence.  This definition may be summarized by the diagram of local ring homomorphisms $R\to \widehat R \gets Q$, where $Q$ is regular and the kernel of $\widehat R \gets Q$ is generated by a $Q$-regular sequence. 

This diagram undoubtedly motivated the following definition from \cite{AGP}. A module $M$ over a local ring $R$ is said to have  
{\em finite complete intersection dimension}, written $\Cdim_RM<\infty$, if there exists a diagram of local ring homomorphisms (namely {\em quasi deformation}) $R\to R' \gets Q$  such that $R\to R'$ is flat, the kernel of 
$R'\gets Q$ is generated by a $Q$-regular sequence, and $M\otimes_RR'$ has finite projective dimension over $Q$. We refer the reader to \cite{AGP} for details on complete intersection dimension.

In view of Proposition \ref{qperfect}, it is natural to ask the following question.

\begin{ques}\label{q} Let $R\to R'$ be a flat map of local rings, and $M$ be an $R$-module.
If $\qpd_{R'}M\otimes_RR'<\infty$, then is $\qpd_RM<\infty$\,?  In particular, if $R$ is a local complete intersection, then does every finitely generated $R$-module have finite quasi-projective dimension? More generally, does an $R$-module $M$ with $\Cdim_RM<\infty$ satisfy $\qpd_RM<\infty$\,?
\end{ques}

We note that the converse to this last question is false.  Indeed, for a local ring $(R,\m,k)$
one always has that $\qpd_Rk<\infty$.  However $\Cdim_Rk<\infty$ only if $R$ is a complete intersection ring.

We do not have an answer to Question \ref{q}, but the following remark lends support.

\begin{rem}\label{cidim}
Let $R$ be a complete intersection ring, and $M$ be an $R$-module.  Then from Corollary \ref{36} one has that $\qpd_{\widehat R}M\otimes_R\widehat{R}<\infty$.  

More generally, suppose that $R$ is a local ring and $M$ an $R$-module 
with $\Cdim_RM<\infty$. Then there exists a quasi deformation $R\to R' \leftarrow Q$ such that $\pd_QM\otimes_RR'<\infty$. Thus by Proposition \ref{qperfect}, we have $\qpd_{R'}(M\otimes_RR')<\infty$.
\end{rem}

Dwyer, Greenlees and Iyengar \cite{DGI} defined a {\em virtually small} complex as a complex $X$ of $R$-modules whose thick closure in the derived category of $R$ contains a nonexact perfect complex.

\begin{prop}\label{20}
Let $M\ne0$ be a finitely generated $R$-module of finite quasi-projective dimension.
Then $M$ is virtually small as an $R$-complex.
\end{prop}

\begin{proof}
By Proposition \ref{14}, there is a perfect complex $P$ which is a quasi-projective resolution of 
$M$. Then $P$ is nonexact as $M\ne0$. In general, every bounded complex $X$ of $R$-modules belongs to the thick closure of its homology complex $\H(X)$; see \cite[3.10]{DGI} for instance. Hence $P$ belongs to the thick closure of $\H(P)$, which coincides with the thick closure of $M$.
\end{proof}

In Remark \ref{r}, we will see that the converse of the above proposition does not necessarily hold.

According to \cite[Theorem 5.2]{P}, a local ring $R$ is a complete intersection if and only if every  nonexact bounded complex of finitely generated $R$-modules is virtually small.
Combining this with Proposition \ref{20} and Corollary \ref{36}, we are naturally lead to the following question, asking whether the converse of the first statement of Question \ref{q} holds.

\begin{ques}\label{q3}
Let $R$ be a  local ring.
Suppose that every finitely generated $R$-module has finite quasi-projective dimension.
Then is $R$ a complete intersection?
\end{ques}

Recently, Question \ref{q3} was considered by Briggs, Grifo and Pollitz \cite{BGP}. They answered it affirmatively in the special case when $R$ is an equipresented local ring; see \cite[Theorem A]{BGP}. In Theorem \ref{t} we will prove that if every finitely generated $R$-module has finite quasi-projective dimension, then such a ring $R$ has to be a so-called AB ring.

\section{The Auslander--Buchsbaum Formula and Depth Formula}

In this section we establish the Auslander--Buchsbaum formula for modules of finite quasi-projective dimension.  We also show that the depth formula holds for such modules. We first show that over a local ring all modules admit minimal quasi-projective resolutions realizing the quasi-projective dimension.

\begin{prop}\label{16}
Let $R$ be a local ring.
Let $M\ne0$ be a finitely generated $R$-module with $\qpd_RM<\infty$.
Then there exists a finite minimal quasi-projective resolution $F$ of $M$ such that 
$\qpd_RM=\sup F-\hsup F$.
\end{prop}

\begin{proof} By Proposition \ref{14} there exists a perfect complex $P$ that is a quasi-projective resolution of $M$ and $\qpd_RM=\sup P-\hsup P$. It follows from 
\cite[Proposition 1.1.2(iv)]{Av} that there is an isomorphism (in the category) of complexes 
$P\cong F\oplus E$ such that $F$ and $E$ are complexes of finitely generated free $R$-modules with 
$F$ minimal and $E$ split exact. Thus $\H_i(P)\cong\H_i(F)\oplus\H_i(E)=\H_i(F)$ for each $i$.
Therefore $F$ is quasi-isomorphic to $P$, and so a minimal finite quasi-projective resolution 
of $M$.  Finally we have $\qpd_RM\le\sup F-\hsup F\le\sup P-\hsup P=\qpd_RM$, and so equality must hold.
\end{proof}

\begin{rem}
The lengths of minimal quasi-projective resolutions may not be the same.
For example, let $(R,\m)$ be a local ring with $\edim R\geq 2$ and $\m^2=0$.
Then the Koszul complex $K$ of $R$ with respect to the minimal system of generators of $\m$ is a minimal quasi-projective resolution of $k$ of length $\edim R$.
Since $\m^2=0$, a truncation $(0\to K_1 \to K_0 \to 0)$ of $K$ is a minimal quasi-projective resolution of $k$ as well.
\end{rem}

\begin{lem}\label{depth}
Let $(R,\m,k)$ be a local ring. Let $M\ne0$ be an $R$-module with $\qpd_RM<\infty$.
One then has the inequality $\depth M\le\depth R$.
\end{lem}

\begin{proof}
Suppose $t:=\depth R<\depth M$.
Then there exists a sequence $\xx=x_1,\dots,x_t$ in $R$ which is regular on both $R$ and $M$.
Replacing $R$ and $M$ with $R/(\xx)$ and $M/\xx M$ respectively, we may assume $\depth R=0$ by Proposition \ref{7}(2).
According to Proposition \ref{16}, one can take a finite minimal quasi-projective resolution $F$ of $M$.  Let $s=\sup F$. Applying the functor $\Hom_R(k,-)$ to the exact sequence 
$0\to \H_s(F)\to F_s \xrightarrow{\partial_s} F_{s-1}$ we get an exact sequence
$$
0\to \Hom_R(k,\H_s(F))\to \Hom_R(k,F_s)\xrightarrow{\Hom(k,\partial_s)} \Hom_R(k,F_{s-1}).
$$
As $F$ is minimal, $\Hom(k,\partial_s)=0$ and therefore $\Hom_R(k,\H_s(F))\cong \Hom_R(k,F_s)$.
Recall $\depth M>\depth R=0$ and $\H_s(F)$ is zero or a direct sum of copies of $M$. In either
case we get $\Hom_R(k,\H_s(F))=0$, and therefore $\Hom_R(k,F_s)=0$.
As $F_s\ne0$, we obtain $\Hom_R(k,R)=0$, that is, $\depth R>0$, which is a contradiction.
\end{proof}

\subsection{The Auslander--Buchsbaum Formula and its consequences}

We prove the Auslander--Buchsbaum formula for modules of finite quasi-projective dimension.

\begin{thm}\label{AB}
Let $R$ be a local ring, and $M$ be a finitely generated $R$-module of finite quasi-projective dimension. Then
$$
\qpd_RM=\depth R-\depth_RM.
$$
\end{thm}

\begin{proof}
Applying Proposition \ref{16}, we choose a minimal quasi-projective resolution $F$ of $M$ with $s=\sup F$, $h=\hsup F$ and $\qpd_RM=s-h$.
Write $\H_h(F)=M^{\oplus a}$ with $a>0$.
Set $C=\Coker\partial_{h+1}$.
The sequence $0\to F_s \xrightarrow{\partial_s} \cdots\xrightarrow{\partial_{h+2}} F_{h+1}\xrightarrow{\partial_{h+1}} F_h \to C \to 0$ is exact.
As $F$ is minimal, this shows $\pd_RC=s-h$.
The Auslander--Buchsbaum formula yields $s-h=\depth R-\depth C$.
Thus, it suffices to show that $\depth C=\depth M$.
Put $u=\depth M$.
By Lemma \ref{depth} we can choose a sequence $\xx=x_1,\dots,x_u$ of elements of $R$ which is regular on $R$ and $M$.
Letting $G=(0\to F_s\xrightarrow{\partial_s}\cdots\xrightarrow{\partial_{h+1}}F_h\to0)$ be the truncated complex with $\deg F_h=0$, we get a quasi-isomorphism $G\to C$.
From Proposition \ref{7}(2) and its proof, we observe that $F\otimes_RR/(\xx)$ is a quasi-projective resolution of $M/\xx M$ over $R/(\xx)$ with $\hsup(F\otimes_RR/(\xx))=h$.
It follows that $\H_i(\xx,C)\cong\Tor_i^R(C,R/(\xx))\cong\H_i(G\otimes_RR/(\xx))=\H_{i+h}(F\otimes_RR/(\xx))=0$ for all $i>0$, which implies that the sequence $\xx$ is $C$-regular.
There is an injection $(M/\xx M)^{\oplus a}=\H_h(F/\xx F)\to\Coker(\overline{\partial_{h+1}})=C/\xx C$.
Since $M/\xx M$ has depth $0$, so does $C/\xx C$.
Therefore, $\depth C=u=\depth M$.
\end{proof}

Theorem \ref{AB} refines (2), (3) and (4) of Proposition \ref{Syz} for modules over a local ring.

\begin{cor}\label{13}
Let $R$ be a local ring.
\begin{enumerate}[\rm(1)]
\item
Let $M,N$ be finitely generated $R$-modules both of finite quasi-projective dimension.
Then $$\qpd_R(M\oplus N)=\sup\{\qpd_RM,\,\qpd_RN\}.$$
\item
Let $M$ be a nonzero finitely generated $R$-module.
\begin{enumerate}[\rm(a)]
\item
If $\qpd_RM<\infty$, then $\qpd_R(M\oplus F)=\qpd_RM$ for all finitely generated free $R$-modules $F$.
\item
One has $\qpd_R(\syz M)\le\sup\{\qpd_RM-1,\,0\}$, and the equality holds if the right-hand side is finite.
\end{enumerate}
\end{enumerate}
\end{cor}

\begin{proof}
(1) We may assume $M\ne0\ne N$.
We then have $\qpd(M\oplus N)<\infty$ by Proposition \ref{Syz}(2).
Theorem \ref{AB} yields
\begin{align*}
\qpd(M\oplus N)
&=\depth R-\depth(M\oplus N)
=\depth R-\inf\{\depth M,\,\depth N\}\\
&=\sup\{\depth R-\depth M,\,\depth R-\depth N\}
=\sup\{\qpd M,\,\qpd N\}.
\end{align*}

(2a) The assertion immediately follows by putting $N:=F$ in (1).

(2b) We may assume that $M$ is a nonfree $R$-module with $\qpd M<\infty$.
Then Proposition \ref{Syz}(4) implies $\qpd(\syz M)<\infty$, and we obtain
\begin{align*}
\qpd(\syz M)
&=\depth R-\depth(\syz M)=\depth R-\inf\{\depth M+1,\,\depth R\}\\
&=\sup\{\depth R-\depth M-1,\,0\}
=\sup\{\qpd M-1,\,0\}
\end{align*}
by virtue of Theorem \ref{AB}.
\end{proof}

Here is an immediate and interesting consequence of Theorem \ref{AB}, which is used later.

\begin{cor}\label{23}
Let $R$ be a local ring and let $M$ be a nonfree finitely generated $R$-module.
Assume $\qpd_RM<\infty$ and $\depth_RM=\depth R$.
Then $M^{\oplus a}$ is a second syzygy for some $a>0$.
In particular, $M$ is a first syzygy.
\end{cor}

\begin{proof}
By Proposition \ref{16} there is a finite minimal quasi-projective resolution $F$ of $M$ with
$\qpd_RM=\sup F-\hsup F$.  Set $s=\sup F$. Since $M$ is nonfree, we have $s>0$. The assumption on depth of the corollary and 
Theorem \ref{AB} imply $s=\hsup F$. There is an exact sequence $0\to\H_s(F)\to F_s\to F_{s-1}$ and 
$\H_s(F)=M^{\oplus a}\ne0$ for some $a>0$.  Thus $M^{\oplus a}$ is a second syzygy. By choosing an embedding $M\to M^{\oplus a}$, we obtain an embedding $M\to F_s$, and so $M$ is a first syzygy.
\end{proof}

Applying the above corollary, we get the following result.

\begin{cor}\label{c}
Let $(R,\m)$ be a local ring, $M$ be a finitely generated $R$-module, and $N$ be a maximal non-free summand of $M$. If $\qpd_RM<\infty$, then $\soc R\subseteq\ann N$.
\end{cor}

\begin{proof}
If $\depth R>0$, then $\soc R=0$ and the conclusion clearly holds.
So we may assume $\depth R=0$.
By Corollary \ref{23} the $R$-module $M$ is a submodule of a free module, and so is $N$.
Hence there exists a monomorphism $f=(f_1,\dots,f_n):N\hookrightarrow R^{\oplus n}$ with $n>0$.
Pick any element $r\in\soc R$.
Suppose that $f_i(rx)=r\cdot f_i(x)\ne0$ for some $1\le i\le n$ and $x\in N$.
Then $f_i(x)$ is a unit of $R$.
Hence the map $f_i:N\to R$ is surjective, and therefore split.
This contradicts the choice of $N$.  Therefore we have $ f_i(rx)=0$ for all $1\le i\le n$ and $x\in N$. Therefore $f(rx)=0$ for all $x\in N$. Since $f$ is injective, $rx=0$ for all $x\in N$.
Thus $r\in\ann N$.
\end{proof}

So far, no module with infinite quasi-projective dimension has appeared.
Therefore one may wonder if all modules have finite quasi-projective dimension, especially in view of the fact that the residue field of a local ring does so (Proposition \ref{18}).
The corollary above gives many examples of modules of infinite quasi-projective dimension.

\begin{exam}\label{e}
Let $k$ be a field and $R=k[x,y]/(x^2,xy,y^2)$.
Then $\qpd_RR/(x)=\infty$ by Corollary \ref{c}.
\end{exam}

\begin{rem}\label{r}
The converse of Proposition \ref{20} does not necessarily hold.
Indeed, let $R$ be as in Example \ref{e}, and set $M=R/(x)\oplus R$.
Then it is clear that $M$ is virtually small, but has infinite quasi-projective dimension by Corollary \ref{c}.
\end{rem}

Here is another corollary of Theorem \ref{AB}.

\begin{cor}\label{epqd}
Let $M$ be a finitely generated module over a (not necessarily local) ring $R$ with $\pd_RM<\infty$.
Then $\qpd_RM=\pd_RM$.
\end{cor}

\begin{proof}
We may assume $M\ne0$.
Remark \ref{8}(3) implies $\qpd_RM\le\pd_RM<\infty$.
Choose a prime ideal $\p$ of $R$ such that $\pd_{R_\p}M_\p=\pd_RM$.
Proposition \ref{7}(1) implies $\qpd_{R_\p}M_\p\le\qpd_RM$.
As $R_\p$ is a local ring, we can apply Theorem \ref{AB} to obtain
$$
\pd_RM\ge \qpd_RM\ge\qpd_{R_\p}M_\p=\depth R_\p-\depth_{R_\p}M_\p=\pd_{R_\p}M_\p=\pd_RM.
$$
Thus the equality $\qpd_RM=\pd_RM$ holds.
\end{proof}

\subsection{The Depth Formula}

Next we prove that the depth formula holds for modules of finite quasi-projective dimension.

\begin{thm}\label{depthformula}
Let $(R,\m,k)$ be a local ring.
Let $M$ and $N$ be finitely generated $R$-modules.
Suppose that $M$ has finite quasi-projective dimension and $\Tor_{>0}^R(M,N)=0$.
Then
$$
\depth M+\depth N=\depth R+\depth(M\otimes_RN).
$$
\end{thm}

\begin{proof}
The assertion of the theorem is clear if either $M$ or $N$ is free, so we assume that neither $M$ nor $N$ is free. According to Proposition \ref{16}, there exists a minimal quasi-projective resolution $(F,\partial)$ of $M$ with $s=\sup F$, $h=\hsup F$ and $\qpd M=s-h$.
Putting $C=\Coker\partial_{h+1}$, we get an exact sequence 
\[
0\to F_s\xrightarrow{\partial_s}\cdots\xrightarrow{\partial_{h+1}}F_h\to C\to0
\]
This is a minimal free resolution of $C$, so that $\pd_RC=s-h$.
Set $Z_i=\z_i(F)$ and $B_i=\b_i(F)$.
There exist exact sequences
$$
0\to B_i\to Z_i\to M^{\oplus a_i}\to0,\qquad
0\to Z_i\to F_i\to B_{i-1}\to0
$$
for $i\in\Z$ and $a_i\ge 0$.
Hence $Z_i\cong\syz B_{i-1}$.
Using the assumption that $\Tor_{>0}^R(M,N)=0$, we see that
$$
\Tor_j^R(B_i,N)\cong\Tor_j^R(Z_i,N)\cong\Tor_{j+1}^R(B_{i-1},N)
\cong\cdots\cong\Tor_{j+i+1-\inf F}^R(B_{\inf F-1},N)=0
$$
for all $i\ge \inf F$ and $j>0$, where the equality holds as $B_{\inf F-1}=0$.
Hence $\Tor_{>0}^R(Z_i,N)=0$ and $\Tor_{>0}^R(B_i,N)=0$ for every $i\in\Z$.

Note that there is an exact sequence
\begin{equation}\label{1}
0\to\H_h(F)\to C\xrightarrow{\overline{\partial_h}}F_{h-1}\to D\to0,
\end{equation}
where $\H_h(F)=M^{\oplus a_h}\neq 0$, $\Im\overline{\partial_h}=B_{h-1}$ and $D=\Coker\partial_h$.
It follows that $\Tor_{>0}^R(C,N)=0$.
The depth formula for a module of finite projective dimension \cite[Theorem 1.2]{A} implies
\begin{equation}\label{2}
\depth C+\depth N=\depth R+\depth(C\otimes_RN).
\end{equation}
Theorem \ref{AB} implies $\depth R-\depth M=\qpd_RM=s-h=\pd_RC=\depth R-\depth C$, and hence
\begin{equation}\label{3}
\depth C=\depth M.
\end{equation}
Tensoring $N$ with \eqref{1} and noting $\Tor_1^R(B_{h-1},N)=0$, we get exact sequences
\begin{align}
\label{4}&0\to\H_h(F)\otimes_RN\to C\otimes_RN\to B_{h-1}\otimes_RN\to0,\\
\label{6}&\Tor_1^R(D,N)\to B_{h-1}\otimes_RN\to F_{h-1}\otimes_RN\to D\otimes_RN\to0
\end{align}
where $\H_h(F)\otimes_RN=M^{\oplus a_h}\otimes_RN\ne0$ by Nakayama's lemma.
We claim the following.
\begin{equation}\label{5}
\text{If $\depth(M\otimes_RN)=0$, then $\depth M+\depth N=\depth R+\depth(M\otimes_RN)$.}
\end{equation}
Indeed, if $\depth(M\otimes_RN)=0$, then it follows from \eqref{4} that $\depth(C\otimes_RN)=0$, and the equality follows from \eqref{2} and \eqref{3}.

From now on, we prove the assertion of the theorem by induction on $\depth N$.
Assume $\depth N=0$. Then by \eqref{2} we have
$$
0=(\depth R-\depth C)+\depth(C\otimes_RN)=\pd_RC+\depth(C\otimes_RN).
$$
Note that both $\pd_RC$ and $\depth(C\otimes_RN)$ are nonnegative.
Hence $s-h=\pd_RC=\depth(C\otimes_RN)=0$, which implies $s=h$ and $C=F_s$.
There is an exact sequence 
\[
0\to Z_{s-1}/B_{s-1}\to F_{s-1}/B_{s-1}\to F_{s-1}/Z_{s-1}\to0.
\]
We have $Z_{s-1}/B_{s-1}=\H_{s-1}(F)=M^{\oplus a_{s-1}}$, $F_{s-1}/B_{s-1}=D$ and $F_{s-1}/Z_{s-1}\cong B_{s-2}$.
Therefore $\Tor_{>0}^R(D,N)=0$ and from \eqref{4} and \eqref{6} we obtain an exact sequence
$$
0\to\H_s(F)\otimes_RN\to F_s\otimes_RN\xrightarrow{\partial_s\otimes_RN}F_{s-1}\otimes_RN\to D\otimes_RN\to0.
$$
Applying $\Hom_R(k,-)$ gives an exact sequence
$$
0\to\Hom_R(k,\H_s(F)\otimes_RN)\to\Hom_R(k,F_s\otimes_RN)\xrightarrow{\Hom_R(k,\partial_s\otimes_RN)=0}\Hom_R(k,F_{s-1}\otimes_RN),
$$
which implies $\Hom_R(k,\H_s(F)\otimes_RN)\cong\Hom_R(k,F_s\otimes_RN)$.
Recall that $\H_s(F)\otimes_RN=M^{\oplus a_s}\otimes_RN\ne0$ and $F_s$ is a nonzero free $R$-module.
Since $\depth N=0$, we observe $\depth(M\otimes_RN)=0$.
The assertion now follows from \eqref{5}.

Now suppose $\depth N>0$.
If $\depth(M\otimes_RN)=0$, then by \eqref{5} we are done.
So let $\depth(M\otimes_RN)>0$.
Then we find an element $x\in R$ which is regular on both $N$ and $M\otimes_RN$.
The exact sequence $0\to N\xrightarrow{x}N\to N/xN\to0$ induces exact sequences
\begin{align*}
&0=\Tor_1^R(M,N)\to\Tor_1^R(M,N/xN)\xrightarrow{0}M\otimes_RN\overset{x}{\hookrightarrow}M\otimes_RN,\\
&0=\Tor_i^R(M,N)\to\Tor_i^R(M,N/xN)\to\Tor_{i-1}^R(M,N)=0\text{ for all }i>1.
\end{align*}
Therefore $\Tor_{>0}^R(M,N/xN)=0$.
The induction hypothesis implies $\depth M+\depth N/xN=\depth R+\depth(M\otimes_RN/xN)$.
It remains to note that $\depth N/xN=\depth N-1$ and $\depth(M\otimes_RN/xN)=\depth((M\otimes_RN)/x(M\otimes_RN))=\depth(M\otimes_RN)-1$.
\end{proof}

Applying Theorem \ref{depthformula} together with Remarks \ref{8}(3) and \ref{cidim}, we recover the results of Auslander, Huneke and Wiegand.

\begin{cor}
Let $R$ be a local ring.
Let $M$ and $N$ be finitely generated $R$-modules with $\Tor_{>0}^R(M,N)=0$.
Then $\depth M+\depth N=\depth R+\depth(M\otimes_RN)$ if either of the following holds.
\begin{enumerate}[\rm(1)]
\item
{\rm(Auslander \cite{A})} $M$ has finite projective dimension.
\item
{\rm(Huneke and Wiegand \cite{HW})} $R$ is a complete intersection.
\end{enumerate}
\end{cor}

\section{Finitistic length for quasi-projective resolutions}

We have studied quasi-projective dimension so far, but it is also natural to ask about the infimum of the lengths of quasi-projective resolutions of a given object.
In this section, we investigate this question.
We make the following definition.

\begin{defn}\label{qpld}
Let $\A$ be an abelian category with enough projective objects.
Let $M$ be an object of $\A$.
We define the {\em quasi-projective length} of $M$ in $\A$ by
$$
\qpl_\A M=
\begin{cases}
\inf\{\sup P-\inf P\mid\text{$P$ is a finite quasi-projective resolution of $M$}\} & (\text{if $M\ne0$}),\\
-\infty & (\text{if $M=0$}).
\end{cases}
$$
We simply set $\qpl_RM=\qpl_{\Mod R}M$ for an $R$-module $M$.
\end{defn}

\begin{rem}\label{19}
\begin{enumerate}[(1)]
\item
Let $\A$ be an abelian category with enough projective objects.
Let $M\in\A$.
\begin{enumerate}[(a)]
\item
There are inequalities $\qpd_\A M\le\qpl_\A M\le\pd_\A M$, and the equalities hold if $\pd_\A M<\infty$ and $\A=\mod R$ by Corollary \ref{epqd}.
\item
One has $\qpd_\A M<\infty$ if and only if $\qpl_\A M<\infty$.
\end{enumerate}
\item
Let $R$ be a local ring, and let $M$ be a nonzero finitely generated $R$-module with $\qpl_RM<\infty$.
Then there exists a finite minimal quasi-projective resolution $F$ of $M$ such that $\qpl_RM=\sup F-\inf F$.
This is a consequence of the proof of Proposition \ref{16}.
\end{enumerate}
\end{rem}

Let $R$ be a local ring.
It is easy to deduce from Theorem \ref{AB} that
$$
\depth R=\sup \{\qpd_RM\mid\text{$M$ is a finitely generated $R$-module with $\qpd_RM<\infty$}\}.
$$
Thus it is natural to think about the following numerical invariant.

\begin{defn}
We define the {\em finitistic quasi-projective length} of $R$ by
$$
\finpql R=\sup \{\qpl_RM\mid\text{$M$ is a finitely generated $R$-module with $\qpl_RM<\infty$}\}.
$$
\end{defn}

Recall that the {\em codepth} of a finitely generated module $M$ over a local ring $R$ is defined as $\codepth_RM=\edim R-\depth M$.

\begin{prop}\label{22}
Let $(R,\m,k)$ be a local ring.
\begin{enumerate}[\rm(1)]
\item
For a finitely generated $R$-module $M$, one has the inequality $\qpl_RM\ge\sup\{\height\p\mid\p\in\Min_RM\}$.
\item
There are inequalities $\dim R\le\qpl_Rk\le\finpql R$.
\item 
One has $\dim R=\qpl_R k$ if and only if $R$ is regular.
\item
Assume that $R$ is a quotient of a regular local ring by a regular sequence.
Then $\qpl_RM\le\codepth_RM$ for all finitely generated $R$-modules $M$.
In particular, $\finpql R \le \edim R$.
\item
If $R$ is an Artinian complete intersection, then $\finpql R = \edim R$.
\end{enumerate}
\end{prop}

\begin{proof}
(1) We may assume $\qpl_RM = l<\infty$.
Then we can take a minimal quasi-projective resolution $F$ of $M$ with $l=\sup F-\inf F$.
Localizing $F$ at any $\p\in\Min_RM$, we get a quasi-projective resolution $F_\p$ of $M_\p$.
Since $M_\p$ has finite length over $R_\p$, so does each homology of the complex $F_\p$.
The New Intersection Theorem implies that $l\geq \dim R_\p=\height\p$.

(2) The assertion immediately follows from (1) and by definition.

(3) By using the Direct Summand Theorem \cite{Andre}, $R$ admits a big Cohen--Macaulay module. Now by \cite[Theorem 2.4]{BI}, $R$ is regular. The converse holds since $\dim R\le \qpl_Rk\le \pd_Rk=\dim R$ by (2)
and Remark \ref{19}(1a). 

(4) Choose a regular local ring $Q$ and a $Q$-sequence $\xx$ such that $R=Q/(\xx)$ and 
$\dim Q=\edim R=:e$.
It follows from Proposition \ref{qperfect} and the Auslander--Buchsbaum formula over $Q$ that 
$\qpl_RM\le\pd_QM=e-\depth M=\codepth_RM\le e=\edim R$ for every finitely generated $R$-module $M$.

(5) Let $\qpl_R k=l$.
Then by (4) we have $l\leq \edim R$.
By Remark \ref{19}(2), there exists a minimal quasi-projective resolution $F$ of $k$ such that $l=\sup F-\inf F$.
Therefore $\mathrm{level}^k_R(F) \le l+1$ by the proof of \cite[3.10]{DGI}; see \cite{ABIM} for the definition of a level and the details.
By \cite[Theorem 11.3]{ABIM}, $\mathrm{level}^k_R(F)\geq \edim R - \cx_RF+1$.
Since $F$ is a perfect complex, it has finite projective dimension in the derived category  of $R$. Therefore $\cx_RF=0$ and hence $l\geq \edim R$.
\end{proof}

The following result shows that the equality of the first inequality in Proposition \ref{22}(2) does not necessarily hold.

\begin{thm}\label{28}
Let $(R,\m,k)$ be a local ring.
One has equivalences:
\begin{enumerate}[\rm(1)]
\item
$\finpql R=0\iff\qpl_Rk=0\iff\text{$R$ is a field}$.
\item
$\finpql R=1\iff\qpl_Rk=1\iff\begin{cases}
\text{$R$ is a discrete valuation ring, or}\\
\text{$R$ is not a field but an artinian hypersurface, or}\\
\text{$R$ is not a field but satisfies $\m^2=0$.}
\end{cases}$
\end{enumerate}
\end{thm}

\begin{proof}
(1) If $R$ is a field, then every $R$-module is free and $\finpql R=0$.
By Proposition \ref{22}(2), if $\finpql R=0$, then $\qpl_Rk=0$.
Now, assume $\qpl_Rk=0$.
Then there exists a minimal quasi-projective resolution $F=(0\to F_s\to0)$ of $k$.
We have $R^{\oplus a}\cong F_s\cong\H_s(F)\cong k^{\oplus b}$ for some $a,b>0$.
Taking the annihilators implies $\m=0$, and hence $R$ is a field.

(2) When $R$ is a discrete valuation ring, it is seen from (1) and Proposition \ref{22}(4) that $\finpql R=1$.

When $R$ is an artinian hypersurface which is not a field, Cohen's structure theorem shows $R\cong S/(x^n)$ for some discrete valuation ring $(S,xS,k)$ and $n\ge2$.
It follows from (1) and Proposition \ref{22}(4) that $\finpql R=1$.

When $\m^2=0$ but $R$ is not a field, there is an isomorphism $\m\cong k^{\oplus e}$ with $e=\edim R\ge 1$.
Take a minimal system of generators $x_1,\dots,x_e$ of $\m$.
The perfect complex $F=(0\to R^{\oplus e}\xrightarrow{\left(\begin{smallmatrix}x_1&\cdots&x_e\end{smallmatrix}\right)}R\to0)$ satisfies $\H_0(F)=k$ and $\H_1(F)=\syz\m\cong\syz(k^{\oplus e})\cong\m^{\oplus e}\cong k^{\oplus e^2}$.
Hence $F$ is a quasi-projective resolution of $k$, which implies $\qpl_Rk\le1$.
Let $M$ be an $R$-module with $\qpd_RM<\infty$.
Then $M$ is a submodule of a free module by Corollary \ref{23}, and it is seen that $M$ is a direct sum of copies 
of $R$ and $k$. Hence $\qpl_RM\le1$ by Corollary \ref{13}(2a).
This together with (1) shows $\finpql R=1$.

It follows from (1) and Proposition \ref{22}(2) that if $\finpql R=1$, then $\qpl_Rk=1$.

Suppose $\qpl_Rk=1$.
Then $R$ is not a field by (1), and there exists a minimal quasi-projective resolution 
$F=(0\to F_s\to F_{s-1}\to0)$ of $k$.
Note by Remark \ref{8}(4) that $\H_{s-1}(F)\ne0$.
Letting $\H_s(F)=k^{\oplus a}$ and $\H_{s-1}(F)=k^{\oplus b}$ with $a\ge0$ and $b\ge1$, we get an exact sequence in the lower left, which gives rise to an isomorphism in the lower right.
$$
0\to k^{\oplus a}\to F_s\to F_{s-1}\to k^{\oplus b}\to0,\qquad
k^{\oplus a}\cong\syz^2(k^{\oplus b})\oplus R^{\oplus h}=(\syz^2k)^{\oplus b}\oplus R^{\oplus h}\quad(h\ge0).
$$
The isomorphism in particular shows that $\m$ kills $\syz^2k$.
Write $\syz^2k=k^{\oplus c}$ with $c\ge0$, and we get an exact sequence $0\to k^{\oplus c}\to R^{\oplus e}\xrightarrow{\pi}\m\to0$, where we set $e=\edim R$.
If $c=0$, then $R$ is a discrete valuation ring, and we are done.
If $e=1$, then $R$ is an artinian hypersurface, and again we are done.
So let $c\ge1$ and $e\ge2$.
Let $x_1,\dots,x_e$ be a minimal system of generators of $\m$, and put $G=(0\to R^{\oplus e}\xrightarrow{\left(\begin{smallmatrix}x_1&\cdots&x_e\end{smallmatrix}\right)}R\to0)$.
Then $\H_1(G)=k^{\oplus c}\ne0$ and $\H_0(G)=k$.
Theorem \ref{AB} shows $\depth R=\depth R-\depth k=\qpd_Rk\le1-1=0$.
Put $r=\r(R)=\dim_k(\soc R)>0$.
Applying $\Hom_R(k,-)$ to the exact sequence $0\to k^{\oplus c}\to R^{\oplus e}\xrightarrow{\left(\begin{smallmatrix}x_1&\cdots&x_e\end{smallmatrix}\right)}R$ gives rise to an exact sequence
$$
0\to \Hom_R(k,k^{\oplus c})\to \Hom_R(k,R^{\oplus e})\xrightarrow{\left(\begin{smallmatrix}x_1&\cdots&x_e\end{smallmatrix}\right)=0}\Hom_R(k,R).
$$
Hence $k^{\oplus c}\cong\Hom_R(k,k^{\oplus c})\cong\Hom_R(k,R^{\oplus e})\cong\Hom_R(k,R)^{\oplus e}\cong k^{\oplus re}$, and we get $c=re$.
Note here that the map $\pi:R^{\oplus e}\to\m$ factors through $(R/\soc R)^{\oplus e}$.
This yields a commutative diagram
$$
\xymatrix@R-1pc@C3pc{
& 0\ar[d] & 0\ar[d]\\
& (\soc R)^{\oplus e}\ar@{=}[r]\ar[d] & (\soc R)^{\oplus e}\ar[d]\\
0\ar[r] & k^{\oplus re}\ar[r]\ar[d] & R^{\oplus e}\ar[r]^\pi\ar[d] & \m\ar[r]\ar@{=}[d] & 0\\
0\ar[r] & L\ar[r]\ar[d] & (R/\soc R)^{\oplus e}\ar[r]^-{\overline\pi}\ar[d] & \m\ar[r] & 0\\
& 0 & 0
}
$$
with exact rows and columns.
The left column yields equalities $\ell_R(L)=\ell_R(k^{\oplus re})-\ell_R((\soc R)^{\oplus e})=re-re=0$, which means $L=0$.
Thus the map $\overline\pi:(R/\soc R)^{\oplus e}\to\m$ is an isomorphism.
Note that this map sends an element $(\overline{a_1},\dots,\overline{a_e})\in(R/\soc R)^{\oplus e}$ to the element $\sum_{i=1}^ea_ix_i\in\m$.
It is observed that $\m=(x_1,\dots,x_e)=(x_1)\oplus\cdots\oplus(x_e)$ and $R/\soc R\cong(x_i)$ for each $1\le i\le e$.
Taking the annihilators, we get $\soc R=(0:x_i)$.
Recall that $e\ge2$.
Fix two integers $1\le i,j\le e$ with $i\ne j$.
Then $x_ix_j\in(x_i)\cap(x_j)=0$, and $x_j\in(0:x_i)=\soc R$.
It follows that the maximal ideal $\m=(x_1,\dots,x_e)$ is contained in $\soc R$, which means $\m^2=0$.
Thus, the proof of the theorem is completed.
\end{proof}

We close this section with some natural questions.

\begin{ques}
\begin{enumerate}[(1)]
\item
Let $(R,\m,k)$ be a local ring. Is it true $\finpql R<\infty$\,?
\item
Is there an example of a local ring $(R,\m,k)$ such that $\qpl_Rk<\finpql R$\,?
\end{enumerate}
\end{ques}

\section{Vanishing of Tor and Ext}

In this section we prove some results about vanishing of Ext and Tor for modules of finite quasi-projective dimension.
The results are reminiscent of what holds for modules over a complete intersection.
We begin with the following general remark on a spectral sequence.

\begin{rem}\label{r1}
Consider a convergent spectral sequence $\E_{p,q}^2\Rightarrow\H_{p+q}$.
Let $a,b,n,l$ be integers such that $l\ge b-a$.
Suppose $\E^2_{p,q}=0$ if $q<a$ or $q>b$ or $n\le p\le n+l$.
Then $\E_{n+l+1,a}^\infty\cong\E_{n+l+1,a}^2$.
Indeed, as the differentials $d^k_{p,q}$ are of bidegree $(-k,k-1)$, it follows that all the maps $d^k_{p,q}$ into and out of $\E^k_{n+l+1,a}$ are zero, for all $k\geq 2$.
This is easy to see if $n\le n+l+1-k$, so let $n>n+l+1-k$.
Then $a+k-1>a+l\ge b$.
\end{rem}

Now we can prove the following theorem.

\begin{thm}\label{main1}
Let $M,N$ be $R$-modules and assume $M$ admits a finite quasi-projective resolution $P$.
Set
\begin{align*}
l=&\sup\{\hsup P-\hinf P,\ \hsup(P\otimes_RN)-\hinf P-1 \}\quad\text{and}\\
m=&\sup\{\hsup P-\hinf P,\  -\hinf(\Hom_R(P,N))-\hinf P-1\}
\end{align*}
Then for an integer $n\geq 1$ the following hold.
\begin{enumerate}[\rm(1)]
\item If {$\Tor^R_i(M,N)=0$ for all $n\leq i \leq n+l$}, then {$\Tor^R_i(M,N)=0$} for all $i\geq n$.
\item If {$\Ext^i_R(M,N)=0$ for all $n\leq i \leq n+m$}, then {$\Ext^i_R(M,N)=0$} for all $i\geq n$.
\item If {$\Tor^R_{\gg0}(M,N)=0$}, then {$\Tor^R_i(M,N)=0$} for all $i>\hsup(P\otimes_RN)-\hsup P$.
\item If {$\Ext^{\gg 0}_R(M,N)=0$}, then {$\Ext^i_R(M,N)=0$} for all $i>-\hinf(\Hom_R(P,N))-\hsup P$.
\end{enumerate}
\end{thm}

\begin{proof}
We may assume $M$ and $N$ are non-zero. First we prove (1) and (3).
By Lemma \ref{spseq} there exists a convergent spectral sequence
$$
\E^2_{p,q}\cong\Tor^R_p(\H_q(P),N)\Longrightarrow \H_{p+q}(P\otimes_R N).
$$

(1) Put $a=\hinf P$ and $b=\hsup P$.
Then $\H_q(P)=0$ for $q<a$ or $q>b$.
Also, $\Tor^R_i(M,N)=0$ for $n\leq i \leq n+l$ and $\H_q(P)\cong M^{\oplus c_q}$ for all $q$, $c_q\ge 0$.
Applying Remark \ref{r1}, we obtain isomorphisms $\E^\infty_{n+l+1,a}\cong \E^2_{n+l+1,a}\cong \Tor^R_{n+l+1}(\H_{a}(P),N)$.
Hence $\Tor^R_{n+l+1}(\H_{a}(P),N)$ is a subquotient of $\H_{n+l+1+a}(P\otimes N)$, which is zero since $n+l+1+a>\hsup(P\otimes_RN)$. Therefore $\Tor^R_{n+l+1}(M,N)$=0, as $\H_{a}(P)\cong M^{\oplus {c_{a}}}\ne 0$.
Proceeding in this way, we get $\Tor^R_i(M,N)=0$ for all $i\geq n$.

(3) Set $p=\sup\{i\mid\Tor^R_i(M,N)\ne0\}$. Let $q=\hsup P$.
As $\H_q(P)$ is non-zero and isomorphic to a direct sum of copies of $M$, we have 
$\E^2_{p,q}\cong\Tor^R_p(\H_q(P),N)\neq 0$.
Hence $\E^2_{p,q}$ is the most right and top nonzero point in the $\E^2$ plane and therefore 
$\E^\infty_{p,q}\cong \E^2_{p,q}$.
Thus $\Tor^R_p(\H_q(P),N)$ is a subquotient of $\H_{p+q}(P\otimes_R N)$ and hence 
$\H_{p+q}(P\otimes_R N)\neq 0$. Therefore $p+q\leq \hsup(P\otimes_R N)$ and so 
$p\leq \hsup(P\otimes_RN)-q$.

To see (2) and (4),  note that by Lemma \ref{spseq} there exists a convergent spectral sequence
$$\E^{p,q}_2\cong \Ext^p_R(\H_q(P),N) \Longrightarrow \H^{p+q}\Hom_R(P,N).$$
Now, similar arguments as in (1) and (3) work for (2) and (4) as well.
\end{proof}

The following example shows that one cannot in general assume fewer than $l+1$ consecutive 
vanishing $\Tor^R_i(M,N)$ in Theorem \ref{main1}(1).

\begin{exam} Let $R=k[x,y]/(xy)$ with $k$ a field.
Let $M=R/(x)$ and $N=R/(y)$. The complex $P=(0\to R \xrightarrow{x} R \to 0)$, with $\inf P=0$, is a quasi-projective resolution of $M$. Since $\H_1(P)\cong M$, one has $\hsup P=1$, and so
$\hsup P-\hinf P=1-0=1$. On the other hand, since $x$ is regular on $N$, we have 
$\hsup(P\otimes_RN)=0$, and so $\hsup(P\otimes_RN)-\hinf P-1=0-0-1=-1$. 
Therefore $l=\sup\{1,-1\}=1$. One checks easily that $\Tor^R_i(M,N)=0$ for $i$ odd, 
while $\Tor^R_i(M,N)\neq 0$ for $i$ even. Therefore one vanishing Tor does not force eventually all vanishing Tor. 
\end{exam}

Here is an immediate application of Theorem \ref{main1}.

\begin{cor}\label{17}
Let $M,N$ be $R$-modules, and $n$ be a positive integer.
If $\Tor_i^R(M,N)=0$ (resp. $\Ext_R^i(M,N)=0$) for all $n\le i\le n+\qpl_RM$, then $\Tor_i^R(M,N)=0$ (resp. $\Ext_R^i(M,N)=0$) for all $i\ge \min\{n, \qpd_RM+1\}$.
\end{cor}

\begin{proof}
We may assume $\qpl_RM<\infty$. Let $P$ quasi-projective resolution of $M$ with 
$\qpl_RM=\sup P-\inf P$.
It is clear that 
$\qpl_RM\ge\sup\{\hsup P-\hinf P,\ \hsup(P\otimes N)-\hinf P-1,\,-\hinf\Hom(P,N)-\hinf P-1\}$.
It follows from (1) and (2) of Theorem \ref{main1} that $\Tor_i^R(M,N)=0$ (resp. $\Ext_R^i(M,N)=0$) for all $i\geq n$.

Next, choose a quasi-projective resolution $P'$ of $M$ with $\qpd_RM=\sup P'-\hsup P'$.
Then note that $\qpd_RM=\sup P'-\hsup P'\ge\hsup(P'\otimes N)-\hsup P'$
and $\qpd_RM=\sup P'-\hsup P'\ge-\hinf\Hom(P',N)-\hsup P'$.
It follows from (3) and (4) of Theorem \ref{main1} that $\Tor_i^R(M,N)=0$ (resp. $\Ext_R^i(M,N)=0$) for all $i>\qpd_RM$.
\end{proof}

Recall that a Gorenstein local ring $R$ is called {\em AB} if for all finitely generated $R$-modules $M$ and $N$ with $\Ext^i_R(M,N)=0$ for all $i\gg0$ one has $\Ext^i_R(M,N)=0$ for all $i > \dim R$.
Now we give a result which supports Question \ref{q3}.
It should also be noted that the techniques in the proofs of \cite[Theorem VII]{DGI} and \cite[Theorem 5.2]{P} are quite different from ours, and none of them works for our result.
For an $R$-module $M$ we denote by $\tr M$ the (Auslander) transpose of $M$.

\begin{thm}\label{t}
Let $(R,\m,k)$ be a local ring of depth $t$.
If $\tr\syz^tk$ has finite quasi-projective dimension, then $R$ is Gorenstein.
In particular, a local ring over which every finitely generated module has finite quasi-projective dimension is an AB ring.
\end{thm}

\begin{proof}
Set $T=\tr\syz^tk$.
Let $(P,\partial)$ be a finite quasi-projective resolution of $T$.  Let $h=\hsup P$ and $C=\Coker\partial_{h+1}$. We have $\pd_RC<\infty$ and there is a monomorphism $\H_h(P)\hookrightarrow C$. Since $\H_h(P)$ is a nonzero direct sum of copies of $T$, we get a monomorphism $T\hookrightarrow C$.
By \cite[Theorem 1.3]{fpd} $R$ is Gorenstein.

If every finitely generated  $R$-module has finite quasi-projective dimension, then it follows from Corollary \ref{17} and Theorem \ref{AB} that $R$ is AB.
\end{proof}

The example below says the converse of the second assertion of Theorem \ref{t} does not hold in general.

\begin{exam}\label{ex}
Let $A=k[x,y]/(x^2,xy,y^2)$ with $k$ a field, and $\E_A(k)$ denote the injective hull of $k$. Consider the trivial extension $R=A\ltimes\E_A(k)$.
The $A$-module $\E_A(k)$ has a minimal free presentation  $A^{\oplus3}\xrightarrow{\left(\begin{smallmatrix}x&0&y\\0&y&-x\end{smallmatrix}\right)}A^{\oplus2}\to \E_A(k)\to0$. From this
we easily see that $R\cong k[x,y,u,v]/(x^2,xy,y^2,xu,yv,xv-yu,u^2,uv,v^2)$.
Note that $R$ is an artinian Gorenstein local ring with $\edim R=4$ and $\dim_kR=6$.
It follows from \cite[Theorem 3.5]{HJ} or \cite[Theorem 3.4]{S} that $R$ is AB.
However, the $R$-module $A$ via the natural surjection $R\twoheadrightarrow A$ is not virtually small by \cite[Example 9.13]{DGI}.
This $R$-module does not have finite quasi-projective dimension by Proposition \ref{20}.
\end{exam}

For the convenience of the reader, we summarize some results of this paper in the following remark.

\begin{rem}
Let $R$ be a local ring.
Consider the following conditions.
\begin{enumerate}[(1)]
\item
The ring $R$ is a complete intersection.
\item
Every finitely generated $R$-module has finite quasi-projective dimension.
\item
The ring $R$ is AB.
\item
The ring $R$ is Gorenstein.
\item
Every nonexact bounded complex of finitely generated $R$-modules is virtually small.
\item
Every nonzero finitely generated $R$-module is virtually small.
\end{enumerate}
Then (1) $\Rightarrow$ (2) holds if $R$ is a quotient of a regular ring by Corollary \ref{36}, while Question \ref{q} asks whether (1) $\Leftrightarrow$ (2) always holds.
Theorem \ref{t} yields (2) $\Rightarrow$ (3).
The implication (3) $\Rightarrow$ (4) holds by definition, while the opposite one (4) $\Rightarrow$ (3) does not hold in general by \cite[Theorem]{JS}.
It follows from \cite[Theorem 5.2]{P} that (1) $\Leftrightarrow$ (5) holds, while (5) $\Rightarrow$ (6) is evident.
Proposition \ref{20} gives (2) $\Rightarrow$ (6).
Example \ref{ex} shows that (3) $\Rightarrow$ (6) does not necessarily hold, and hence (3) $\Rightarrow$ (2) is not always true, either.
\end{rem}

For a module $M$ over a local ring $R$, we denote by $\cx_RM$  the {\em complexity} of $M$. The details can be found in \cite{Av,AGP}.
It is worth reproving the following well-known result as an application of Theorem \ref{main1}.

\begin{cor}\label{lbnd}
Let $R$ be a local ring, and let $M$ be a finitely generated $R$-module.
Put $c=\cx_RM+\Cdim_RM$, and let $N$ be an $R$-module and $n\geq 1$ be an integer.
If $\Tor^R_i(M,N)=0$ (resp. $\Ext^i_R(M,N)=0$) for all $n\leq i \leq n+c$, then $\Tor^R_i(M,N)=0$ (resp. $\Ext^i_R(M,N)=0$) for all $i>\Cdim_R M$.
\end{cor}

\begin{proof}
We only prove the result for Tor; a similar argument applies for Ext.
We may assume $\Cdim M<\infty$.
By \cite[Theorem 5.10]{AGP} there is a quasi-deformation $R\to R'\leftarrow Q $ with $\pd_QM'=c$, where $M'=M\otimes_RR'$. We may assume $R=R'$.
Let $P$ be a minimal free $Q$-resolution of $M$.
Then $P':=P\otimes_QR$ is a quasi-projective resolution of $M$ by Proposition \ref{qperfect}.
Since $\H_i(P')\cong\Tor^Q_i(M,R)$ and $\H_i(P'\otimes_RN)\cong \Tor^Q_i(M,N)$, we have $c\ge\hsup P'-\hinf P'$ and $c\ge\hsup(P'\otimes_RN)-\hinf P'-1$.
As $\Cdim_RM=\depth R-\depth M$, the result follows by Theorems \ref{main1}(1), \ref{AB} and Corollary \ref{17}.
\end{proof}

In the following, we recall definition of Gorenstein dimension which we use in the rest of this section.

\begin{defn}
Let $M$ be a finitely generated $R$-module.
\begin{enumerate}[(1)]
\item
We call $M$ \emph{totally reflexive} if the natural homomorphism $M\to M^{**}$ is an isomorphism and $\Ext_R^i(M,R)=0=\Ext_R^i(M^*,R)$ for all $i>0$.
The infimum of nonnegative integers $n$ such that there exists an exact sequence $0 \to G_{n} \to \dots  \to G_{0} \to M\to 0$ with each $G_{i}$ totally reflexive, is called the \emph{Gorenstein dimension} (or {\em G-dimension}) of $M$, and we write $\Gdim_RM=n$.
\item
Suppose that $M$ has finite G-dimension.
\begin{enumerate}[(a)]
\item
A \emph{complete resolution} (or \emph{Tate resolution}) of $M$ over $R$ is a diagram $T\xrightarrow{\nu}P\xrightarrow{\pi}M$, where $P$ is a projective resolution of $M$ in $\mod R$, and $(T,\partial^T)$ is an exact complex of projective modules in $\mod R$ such that $T^*$ is exact and $\nu$ is a chain map with $\nu_i$ bijective for all $i\gg0$.
Note that a complete resolution of $M$ exists if and only if $\Gdim_RM<\infty$; see \cite[Theorem 3.1]{AM}.
For each $i<0$ we set $\syz^iM=\Im(\partial^T_i)$, which is called the $(-i)$th {\em cosyzygy} of $M$.
\item
Let $N$ be an $R$-module (not necessarily finitely generated).
The $i$th \emph{Tate homology} (respectively, \emph{Tate cohomology}) of $M$ and $N$ is defined by $\ctor^R_i(M,N)=\H_i(T\otimes_RN)$ (respectively, $\cext^i_R(M,N)=\H^i(\Hom_R(T,N))$).
\end{enumerate}
\end{enumerate}
\end{defn}

\begin{rem}\label{21}
\begin{enumerate}[(1)]
\item We recall that when $R$ is local and $\Gdim_RM<\infty$, then $\Gdim_RM=\depth R-\depth M$.  Every finitely generated module over a Gorenstein ring has finite G-dimension.  Thus a finitely generated module over a Gorenstein ring is totally reflexive if and only if it is maximal Cohen--Macaulay. 
\item Let $M$ be an $R$-module of finite G-dimension.
Let $T\to P\to M$ be a complete resolution of $M$.
Then one can easily get a complete resolution $T[-1]\to P'\to\syz M$ of $\syz M$, where $P'=(\cdots\to P_2\to P_1\to0)$ is a truncated complex with $P_1$ in degree $0$.
Hence $\ctor_i^R(\syz M,N)\cong\ctor_{i+1}^R(M,N)$ and $\cext_R^i(\syz M,N)\cong\cext_R^{i+1}(M,N)$ for all $i\in\Z$ and $N\in\Mod R$.
We refer the reader to \cite{AM} for more details and properties of complete resolutions and Tate (co)homology.
\end{enumerate}
\end{rem}

\begin{lem}\label{39}
Let $M$ and $N$ be non-zero $R$-modules with $\Gdim_RM<\infty$ and $\qpd_RN<\infty$.
Let $(P,\partial)$ be a finite quasi-projective resolution of $N$.
Let $n>\Gdim_RM$ and $t=\hsup P-\hinf P$.
\begin{enumerate}[\rm(1)]
\item
If $\Tor_i^R(M,N)$ (resp. $\Ext_R^i(M,N)$) vanishes for all $n\le i\le n+t-1$, then it vanishes for $i=n+t+1$.
\item
If $\Tor_i^R(M,N)$ (resp. $\Ext_R^i(M,N)$) vanishes for all $n+2\le i\le n+t+1$, then it vanishes for $i=n$.
\end{enumerate}
\end{lem}

\begin{proof} Set $g=\Gdim_RM\in\N$.
Replacing $M$ with $\syz^gM$, we may assume that $M$ is totally reflexive.
If $t=0$, then $\pd_RN<\infty$ and $\Tor^R_{>0}(M,N)=\Ext_R^{>0}(M,N)=0$ by \cite[Proposition (4.12)]{AB}.
So let $t>0$. Set $h=\hsup P$, $l=\hinf P$, and $C=\Coker\partial_{h+1}$.
Since $\pd_RC<\infty$, similarly as above $\Tor^R_{>0}(M,C)=\Ext_R^{>0}(M,C)=0$.
There are exact sequences
$$
0\to H_h\to C\to B_{h-1}\to0,\quad
\begin{cases}
0\to B_i\to Z_i\to H_i\to0,\\
0\to Z_i\to P_i\to B_{i-1}\to0
\end{cases}
(i\in\Z),
$$
where $Z_i=\z_i(P)$, $B_i=\b_i(P)$ and $H_i=\H_i(P)$.
We use the fact that each $P_i$ is projective and each $H_i$ is a direct sum of copies of $N$.

(1) First we show the assertion for Tor.
We have $0=\Tor^R_n(M,H_h)=\Tor^R_{n+1}(M,B_{h-1})$, from which we get 
$0=\Tor^R_{n+1}(M,Z_{h-1})=\Tor^R_{n+2}(M,B_{h-2})$ if $t\ge2$, from which we get 
$0=\Tor^R_{n+2}(M,Z_{h-2})=\Tor^R_{n+3}(M,B_{h-3})$ if $t\ge3$.
Iterating this procedure gives, and using the fact that $Z_l$ is projective, we get  
$\Tor^R_{n+t+1}(M,H_l)=\Tor^R_{n+t}(M,B_l)=0$, and hence 
$\Tor^R_{n+t+1}(M,N)=0$.

Next we show the assertion for Ext.
We have, since $Z_l$ is projective, 
$0=\Ext_R^n(M,H_l)=\Ext_R^{n+1}(M,B_l)=\Ext_R^{n+2}(M,Z_{l+1})$, 
from which we get $0=\Ext_R^{n+2}(M,B_{l+1})=\Ext_R^{n+3}(M,Z_{l+2})$ if $t\ge2$, 
from which we get $0=\Ext_R^{n+3}(M,B_{l+2})=\Ext_R^{n+4}(M,Z_{l+3})$ if $t\ge3$.
Iterating this procedure gives  
$\Ext_R^{n+t+1}(M,H_h)=\Ext_R^{n+t}(M,B_{h-1})=0$, and hence $\Ext_R^{n+t+1}(M,N)=0$.

(2) First we show the assertion for Tor.
We have $0=\Tor^R_{n+t+1}(M,H_l)=\Tor^R_{n+t}(M,B_l)=\Tor^R_{n+t-1}(M,Z_{l+1})$, from which we get 
$0=\Tor^R_{n+t-1}(M,B_{l+1})=\Tor^R_{n+t-2}(M,Z_{l+2})$ if $t\ge2$, from which we get 
$0=\Tor^R_{n+t-2}(M,B_{l+2})=\Tor^R_{n+t-3}(M,Z_{l+3})$ if $t\ge3$.
Iterating this procedure gives $\Tor^R_n(M,H_h)=\Tor^R_{n+1}(M,B_{h-1})=0$, and hence 
$\Tor^R_n(M,N)=0$.

Next we show the assertion for Ext.
We have $0=\Ext_R^{n+t+1}(M,H_h)=\Ext_R^{n+t}(M,B_{h-1})$, from which we get 
$0=\Ext_R^{n+t}(M,Z_{h-1})=\Ext_R^{n+t-1}(M,B_{h-2})$ if $t\ge2$, from which we get 
$0=\Ext_R^{n+t-1}(M,Z_{h-2})=\Ext_R^{n+t-2}(M,B_{h-3})$ if $t\ge3$.
Iterating this procedure, and using the fact that $Z_l$ is projctive, we have 
$\Ext_R^n(M,H_l)=\Ext_R^{n+1}(M,B_l)=0$, and hence $\Ext_R^n(M,N)=0$.
\end{proof}

\begin{thm}\label{38}
Let $M$ and $N$ be non-zero $R$-modules with $\qpd_RN<\infty$.
Let $n>0$ be an integer, and $P$ a finite quasi-projective resolution of $N$.
If $\Gdim_RM<n$ and $\Tor_i^R(M,N)$ (resp. $\Ext_R^i(M,N)$) vanishes for all 
$n\le i\le n+\hsup P-\hinf P$, then it vanishes for all $i>\Gdim_RM$.
\end{thm}

\begin{proof}
It is enough to show the assertions for Tor; the ones for Ext are shown similarly.
Put $t=\hsup P-\hinf P$ and $m=\Gdim M+1$.
Lemma \ref{39}(1) implies $\Tor^R_i(M,N)=0$ for all $n\le i\le n+t+1$.
Hence $\Tor^R_i(M,N)=0$ for all $n+1\le i\le (n+1)+t$, and applying Lemma \ref{39}(1) again gives 
$\Tor^R_i(M,N)=0$ for all $n\le i\le n+t+2$.
Iteration of this procedure shows $\Tor^R_{\ge n}(M,N)=0$.
If $m\ge n$, then $\Tor^R_{\ge m}(M,N)=0$.
If $m\le n-1$, then $\Gdim M=m-1<n-1$, and the vanishing of $\Tor^R_i(M,N)$ for all $(n-1)+2\le i\le (n-1)+t+1$ and Lemma \ref{39}(2) imply $\Tor^R_{\ge(n-1)}(M,N)=0$.
If $m\le n-2$, then $\Gdim M=m-1<n-2$, and the vanishing of $\Tor^R_i(M,N)$ for all $(n-2)+2\le i\le (n-2)+t+1$ and Lemma \ref{39}(2) imply $\Tor^R_{\ge(n-2)}(M,N)=0$.
Iteration of this procedure shows $\Tor^R_{\ge m}(M,N)=0$.
\end{proof}

The following corollary is immediate from the above theorem.

\begin{cor}\label{41}
Let $M$ and $N$ be $R$-modules such that $\qpd_RN<\infty$.
If there exists an integer $n>0$ with $\Gdim_RM<n$ such that $\Tor_i^R(M,N)$ (resp. $\Ext^i_R(M,N)$) vanishes for all $n\le i\le n+\qpl_RN$, then it vanishes for all $i>\Gdim_RM$.
\end{cor}

The following result gives a criterion for modules of finite quasi-projective dimension to have finite Gorenstein dimension. This is the main point of \cite{AC}.

\begin{prop}\label{mstar} Let $M$ be a finitely generated $R$-module with finite quasi-projective dimension.
If $\Ext^{\gg 0}_R(M,R)=0$ then $\Gdim_RM<\infty$. In particular, if $\Ext^i_R(M,R)=0$ for all $i>0$, then $M$ is a totally reflexive $R$-module, and one has $\qpd_RM^* < \infty$ and $\qpd_R\Omega^iM<\infty$ for all $i\in \Z$.
\end{prop} 
\begin{proof} By using Proposition \ref{Syz}(4), we can replace $M$ with $\Omega^nM$ for some $n \gg 0$ and assume $\Ext^{>0}_R(M,R)=0$. 
We show that $M$ is a totally reflexive module.
Proposition \ref{14} gives a perfect complex $P$ which is a quasi-projective resolution of $M$. 
By Lemma \ref{spseq}, there exists a spectral sequence $\Ext^p_R(\H_q(P),R)\Longrightarrow \H^{p+q}(P^*)$.
Since $\Ext^{>0}_R(M,R)=0$, we get an isomorphism $(\H_q(P))^*\cong \H^q(P^*)$ for all $q\geq 0$.
Since $(\H_q(P))^*$ is a direct sum of copies of $M^*$ and $P^*$ is a complex of projective $R$-modules, it follows that $P^*$ is a quasi-projective resolution of $M^*$. Hence $\qpd_RM^*<\infty$.
Next, by applying Lemma \ref{spseq} once more, there exists a spectral sequence 
$$\E^{p,q}_2\cong \Ext^p_R(\H_q(P^*),R)\Longrightarrow \H^{p+q}(P^{**}).$$
The natural isomorphism $P\to P^{**}$ implies that $\H_i(P)\cong \H_i(P^{**})$. In the last spectral sequence, since $\E^{0,\hinf P^*}_2 \cong \E^{0,\hinf P^*}_\infty$, we get $\Hom_R(\H_{\hinf P^*}(P^*),R)\cong \H_{\hsup P}(P)$. Since the edge homomorphisms $\H_i(P) \to \Hom_R(\H_{-i}(P^*),R)$ are induced from the natural maps, we have $\H_{\hsup P}(P)$ is reflexive and therefore $M$ is reflexive as well. Hence $\Hom_R(\H_{-i}(P^*),R)\cong \H_i(P)$ for all $i$. This implies that $\Ext^{>0}_R(M^*,R)=0$. Therefore $M$ is totally reflexive.

By Proposition \ref{Syz}(4) we have $\qpd_R\Omega^iM<\infty$ for all $i\geq 0$. For $i<0$, one has $\Omega^iM \cong (\Omega^{-i}(M^*))^*$. Therefore by the last argument and Proposition \ref{Syz}(4) we get $\qpd_R\Omega^iM < \infty$.
\end{proof}

\begin{prop}\label{Tate}
Let $M$ and $N$ be $R$-modules.
Assume that $M$ has finite G-dimension, and that either $M$ or $N$ has finite quasi-projective dimension.
Then the following equivalences hold.
\begin{enumerate}[\rm(1)]
\item
$\cext^{\gg0}_R(M,N)=0\iff
\cext^{\ll0}_R(M,N)=0\iff
\cext^i_R(M,N)=0\text{ for all $i \in \Z$}$.
\item
$\ctor_{\gg0}^R(M,N)=0\iff
\ctor_{\ll0}^R(M,N)=0\iff
\ctor_i^R(M,N)=0\text{ for all $i \in \Z$}$.
\end{enumerate}
\end{prop}

\begin{proof}
We only prove that $\cext^{\ll0}_R(M,N)=0$ implies $\cext^i_R(M,N)=0$ for all $i \in \Z$; the other implications are either obvious or similarly shown.
We may assume $M\ne0\ne N$.
In view of Remark \ref{21} and Proposition \ref{Syz}(4), we may also assume that $M$ is totally reflexive.
Let $t\in\Z$ be with $\cext^i_R(M,N)=0$ for all $i\leq t$.

First, assume $\qpd M<\infty$ and take a quasi-projective resolution of $M$ of length $l$.
Let $t-l\leq i \leq t$ and fix $j\geq \max\{l-t+1 ,0\}$.
Then $i+j>0$ and $0=\cext^{i}_R(M,N)\cong \cext^{i+j}_R(M,\syz^jN)\cong \Ext^{i+j}_R(M,\syz^jN)$.
It follows that for $l+1$ consecutive $i$ we have $\Ext^i_R(M,\Omega^jN)=0$. Theorem \ref{main1}(2) implies that for all $k\ge t-l+j$ one has $0=\Ext_R^k(M,\syz^jN)\cong\cext_R^k(M,\syz^jN)\cong\cext_R^{k-j}(M,N)$, and thus $\cext^{i}_R(M,N)=0$ for all $i\geq t -l$.
Therefore, $\cext_R^i(M,N)=0$ for all $i\in\Z$.

Next, assume $\qpd N<\infty$ and take a quasi-projective resolution of $N$ of length $r$.
Using the isomorphism $\cext^i_R(M,N)\cong \cext^{i+j}_R(M,\syz^jN)$ and Proposition \ref{Syz}(4), we may replace $N$ with $\syz^jN$ where $j\ge \max\{r-t+1 ,0\}$.
Then by Theorem \ref{38} we get $\Ext^{i}_R(M,N)=0$ for all $i>0$, and therefore $\cext^i_R(M,N)=0$ for all $i\in\Z$.
\end{proof}

Now we obtain a result on symmetry in vanishing of Ext modules.

\begin{thm}\label{24}
Let $R$ be a Gorenstein ring.
Let $M$ and $N$ be finitely generated $R$-modules.
Assume either $\qpd_RM$ or $\qpd_RN$ is finite.
Then $\Ext^{\gg 0}_R(M,N)=0$ if and only if $\Ext^{\gg 0}_R(N,M)=0$.
\end{thm}

\begin{proof}
By symmetry, we may assume $\qpd_RM<\infty$.
By Propositions \ref{Syz}(4) and \ref{7}(1) we may assume that $R$ is local and $M$, $N$ are maximal Cohen--Macaulay.
We have $\cext^{-i-1}_R(M,N) \cong \ctor^R_i(M^*,N)$ for all $i\in\Z$ by \cite[4.4.7]{AB2}.
Putting this together with Proposition \ref{Tate}(1), we get $\Ext_R^{\gg0}(M,N)=0$ if and only if $\cext_R^i(M,N)=0$ for all $i\in\mathbb{Z}$, if and only if $\ctor^R_i(M^*,N)=0$ for all $i\in\mathbb{Z}$. Since $\qpd_RM^*<\infty$ by  Proposition \ref{mstar}, we have $\ctor^R_i(M^*,N)=0$ for all $i\in\mathbb{Z}$ if and only if $\Tor^R_{\gg0}(M^*,N)=0$ by Proposition \ref{Tate}(2). Finally by \cite[Theorem 2.1]{HJ}, $\Tor_{\gg 0}^R(M^*,N)=0$ if and only if $\Ext^{\gg 0}_R(N,M)=0$.
\end{proof}

The above theorem yields the following two corollaries.

\begin{cor}
Let $R$ be a Gorenstein local ring.
Let $M$ and $N$ be finitely generated $R$-modules with $\qpd_RN<\infty$.
Putting $m=\dim R-\depth M$ and $n=\dim R-\depth N$, one has equivalences
$$
\Ext^{\gg 0}_R(M,N)=0\iff
\Ext^{\gg 0}_R(N,M)=0\iff
\Ext^{>m}_R(M,N)=0\iff
\Ext^{>n}_R(N,M)=0.
$$
\end{cor}

\begin{proof}
Theorem \ref{24} shows the first equivalence.
When $\Ext^{\gg 0}(M,N)=0$, $\Ext^{>m}(M,N)=0$ by Corollary \ref{41}(1).
If $\Ext^{\gg 0}_R(N,M)=0$, then $\Ext^{>n}(N,M)=0$ by Corollary \ref{41}(2) and Theorem \ref{AB}.
\end{proof}

\begin{cor}[Avramov and Buchweitz \cite{AB2}]
Let $R$ be a complete intersection.
Let $M$ and $N$ be finitely generated $R$-modules.
Then $\Ext^{\gg 0}_R(M,N)=0$ if and only if $\Ext^{\gg 0}_R(N,M)=0$.
\end{cor}

\begin{proof}
This is a direct consequence of Corollary \ref{36} and Theorem \ref{24}.
\end{proof}

In fact, Avramov and Buchweitz \cite{AB2} also prove that, under the same assumption of the above corollary, $\Tor_{\gg 0}^R(M,N)=0$ if and only if $\Ext^{\gg 0}_R(M,N)=0$. Thus we naturally have the following question.

\begin{ques}\label{q2} Let $R$ be a Gorenstein ring, and let $M, N$ be $R$-modules. Suppose $\qpd_RM<\infty$. Is it true that $\Tor^R_{\gg0}(M,N)=0$ if and only if $\Ext^{\gg 0}_R(M,N)=0$\,?
\end{ques}

Recall that the {\em Auslander--Reiten conjecture} asserts that a finitely generated $R$-module $M$ is projective if $\Ext_R^{>0}(M,M\oplus R)=0$.
This conjecture holds for modules of finite quasi-projective dimension.

\begin{thm}\label{26}
Let $M$ be a finitely generated $R$-module.
\begin{enumerate}[\rm(1)]
\item
If $n:=\qpl_R M<\infty$ and $\Ext_R^i(M,M)=0$ for all $2\le i\le n+1$, then $\pd_R M=n$.
\item
Assume $\qpl_RM<\infty$, or equivalently, $\qpd_RM<\infty$.
If $\Ext_R^i(M,M)=0$ for all $1\le i\le\qpl_RM+1$, then $M$ is projective.
In particular, the Auslander--Reiten conjecture holds for modules of finite quasi-projective dimension.
\end{enumerate}
\end{thm}

\begin{proof}
In both of the two assertions, we may assume $M\ne0$.

(1) There exists a quasi-projective resolution $P$ of $M$ of length $n$.
Put $Z_i=\z_i(P)$, $B_i=\b_i(P)$ and $H_i=\H_i(P)$.
Then $H_i=M^{\oplus a_i}$ with $a_i\ge0$ for all $i$.
In view of Remark \ref{8}(4), we may assume $l:=\hinf P=\inf P$. We have exact sequences
\begin{equation}\label{9}
0\to B_i\xrightarrow{a_i}Z_i\xrightarrow{b_i}H_i\to0,\qquad
0\to Z_i\xrightarrow{c_i}P_i\to B_{i-1}\to0
\end{equation}
for each $i\in\Z$.
The second sequence in \eqref{9} shows that $Z_i\cong\syz_R B_{i-1}$ for all $i\in\Z$.
From the first sequence in \eqref{9} (with $i$ replaced with $i-1$) we get an exact sequence
$$
\Ext_R^{j+1}(B_{i-2},M)=\Ext_R^j(Z_{i-1},M)\to\Ext_R^j(B_{i-1},M)\to\Ext_R^{j+1}(H_{i-1},M)=0
$$
for each $1\le j\le n$ as $\Ext_R^{j+1}(M,M)=0$.
We thus obtain a sequence of surjections
$$
\Ext_R^1(B_{i-1},M)\twoheadleftarrow\Ext_R^2(B_{i-2},M)\twoheadleftarrow\Ext_R^3(B_{i-3},M)\twoheadleftarrow\cdots\twoheadleftarrow\Ext_R^{i-l+1}(B_{l-1},M)=0
$$
for each $l+1\le i\le l+n$, where the equality holds since $B_{l-1}=0$.
Hence $\Ext_R^1(B_{i-1},M)=0$ for all $l+1\le i\le l+n$, and therefore $\Ext_R^1(B_{i-1},M)=0$ for all $i\in\Z$, since $B_k=0$ if $k\le l-1$ or $k\ge l+n$.

Fix any integer $i$.
Using the equality $H_i=M^{\oplus a_i}$, we get $\Ext_R^1(B_{i-1},H_i)=0$.
Applying the functor $\Hom_R(-,H_i)$ to the second sequence in \eqref{9}, we see that the map
$$
\Hom_R(c_i,H_i):\Hom_R(P_i,H_i)\to\Hom_R(Z_i,H_i)
$$
is surjective.
This implies that the canonical surjection $b_i:Z_i\to H_i$ factors through the inclusion map $c_i:Z_i\to P_i$, that is, there is a homomorphism $\rho_i:P_i\to H_i$ with $b_i=\rho_ic_i$.
As $a_i$ is an inclusion map as well, we observe that $\rho_i(B_i)=\rho_ic_ia_i(B_i)=b_ia_i(B_i)=0$, and hence $\rho_i\partial_{i+1}=0$.
We obtain a chain map
$$
\xymatrix@R-1pc{
P\ar[d]^\rho &=(0\ar[r] & P_{l+n}\ar[r]^{\partial_{l+n}}\ar[d]^{\rho_{l+n}} & P_{l+n-1}\ar[r]^{\partial_{l+n-1}}\ar[d]^{\rho_{l+n-1}} & \cdots\ar[r]^{\partial_{l+3}} & P_{l+2}\ar[r]^{\partial_{l+2}}\ar[d]^{\rho_{l+2}} & P_{l+1}\ar[r]^{\partial_{l+1}}\ar[d]^{\rho_{l+1}} & P_l\ar[r]\ar[d]^{\rho_l} & 0)\\
\H(P) &=(0\ar[r] & H_{l+n}\ar[r]^0 & H_{l+n-1}\ar[r]^0 & \cdots\ar[r]^0 & H_{l+2}\ar[r]^0 & H_{l+1}\ar[r]^0 & H_l\ar[r] & 0)
}
$$
and it is easy to verify that each induced map $\H_i(\rho):\H_i(P)=H_i\to H_i=\H_i(\H(P))$ is the identity map.
Therefore $\rho$ is a quasi-isomorphism, which induces an isomorphism
$$\textstyle
P\cong\H(P)=\bigoplus_{i\in\Z}H_i[i]=\bigoplus_{i\in\Z}M^{\oplus a_i}[i]
$$
in the derived category of $R$.
The first term is a bounded complex of projective modules, while the last term contains $M$ as a direct summand.
It follows that $M$ has finite projective dimension.
This is a consequence of the fact that in the derived category of $R$ the perfect complexes form a thick subcategory; see \cite[Example 3.2(ii)]{DGI} for instance.
We have $\pd_R M=\qpl_R M=\qpd_R M=n$ by Remark \ref{19}(1a).

(2) Put $n=\qpl_RM$.
By (1) we have $\pd_RM=n$.
Choose a prime ideal $\p$ of $R$ so that $\pd_{R_\p}M_\p=n$.
As $R_\p$ is a local ring, Nakayama's lemma implies $\Ext_{R_\p}^n(M_\p,M_\p)\ne0$.
In particular, $\Ext_R^n(M,M)\ne0$.
Since $\Ext_R^i(M,M)=0$ for all $1\le i\le n+1$, we must have $n=0$, and therefore $M$ is projective.
\end{proof}

\begin{cor}\label{canonical}
Let $R$ be a Cohen--Macaulay local ring admitting a canonical module $\omega_R$. Then $R$ is Gorenstein if and only if $\qpd_R\omega_R<\infty$.
\end{cor}

\begin{exam}
Let $(R,\m,k)$ be a local ring with $\edim R\geq 2$ and $\m^2=0$.
Since every syzygy of a finitely generated $R$-module is isomorphic to the direct sum of a free $R$-module with a $k$-vector space, it has finite quasi-projective dimension; see Proposition \ref{Syz}(1)(3) and \ref{18}(1).
Therefore $\qpd_R\syz\E(k)<\infty$ where $\E(k)$ is the injective envelope of the residue field $k$.
Since $R$ is not Gorenstein, $\qpd_R \E(k)=\infty$ by Corollary \ref{canonical}.
Thus, even if a syzygy has finite quasi-projective dimension, the original module does not necessarily have finite quasi-projective dimension.
This is also shown by Example \ref{e}; let $R$ be as in it. Then $\qpd_RR/(x)=\infty$ but $\qpd_R(x)=\qpd_Rk<\infty$ by Proposition \ref{18}(1).
\end{exam}

\section{Ideals with free Koszul homology}

In this section, we are interested in ideals $I$ such that the Koszul complex with respect to a system of generators of $I$ is a quasi-projective resolution of $R/I$. We call such an ideal an \emph{ideal with free Koszul homology} or briefly an \emph{FKH ideal}.

Let $\xx=x_1,\dots, x_n$ be a sequence of elements in $R$. We denote with $\K(\xx)$ and $\H_i(\xx)$, the Koszul complex and the $i$th Koszul homology of $R$ with respect to $\xx$, respectively. In the case where $R$ is local and $\xx$ minimally generates an ideal $I$, one has $\H_i(\xx)$ is uniquely determined by $I$ and does not depend on the choice of $\xx$. In this case we write $\K(I)=\K(\xx)$ and $\H_i(I)=\H_i(\xx)$.

\begin{exam}
Let $I$ be a proper ideal of $R$.
\begin{enumerate}[(1)]
\item If $I$ is a maximal ideal of $R$, then $I$ is an FKH ideal.
\item If $I$ is a complete intersection ideal (i.e. $I$ is generated by a regular sequence), then $I$ is FKH.
\item If $I$ is a quasi-complete intersection ideal, then $I$ is FKH; see \cite{AHS}.
\end{enumerate}
\end{exam}

The following example and Proposition \ref{fp} provide other examples of FKH ideals.

\begin{exam}
Let $R=k[w,x,y,z]/(x^2,y^2,w^2,z^2,wz)$ where $k$ is a field.
\begin{enumerate}[(1)]
\item The ideal $I=(x,y)$ is a quasi-complete intersection ideal, and hence, it is FKH.
\item The ideal $J=(w,z)$ is an FKH ideal but not quasi-complete intersection.
\end{enumerate}
\end{exam}

\begin{prop}\label{fp}
Let $(R,\m,k)$ be a fiber product over $k$, and let $\m=I\oplus J$. Suppose $J$ is a principal ideal and $R/J$ is a regular ring.
Then $I$ is an FKH ideal of $R$.
\end{prop}

\begin{proof}
Consider the exact sequence $0\to R \to R/I\oplus R/J \to k \to 0$.
Applying $-\otimes_R\K(I)$ to the last exact sequence, we get a long exact sequence of Koszul homologies
$$
\cdots \to \H_{i+1}(I,k) \to \H_i(I) \to \H_i(I,R/I)\oplus \H_i(I,R/J) \to \H_i(I,k) \to \cdots.
$$
Since $R/J$ is a regular local ring whose maximal ideal is $\m/J \cong I$, we have $\H_i(I,R/J)=0$ for all $i>0$.
As $\H_i(I,R/I) \cong \K_i(I)\otimes_RR/I$ and $\H_i(I,k)\cong \K_i(I)\otimes_Rk$, one checks that the map $\H_i(I,R/J) \to \H_i(I,k)$ is the natural surjection for all $i$ and hence, the connecting homomorphisms are all zero.
Therefore $\H_i(I)$ is isomorphic to a direct sum of copies of $\m/I\cong J \cong R/I$.
\end{proof}

Applying Theorem \ref{main1} to FKH ideals, we obtain the following result which relates quasi-projective dimension with grade.

\begin{thm}\label{main}
Let $I$ be an FKH ideal of $R$.
Let $M$ be an $R$-module and $n\ge1$ an integer.
Set $g=\min \{ \grade I, \grade(I,M) \}$.
\begin{enumerate}[\rm(1)]
\item
One has $\qpd_R(R/I)=\grade I$.
\item
If $\Tor^R_i(R/I,M)=0$ for all $n\leq i \leq e-g+n$, then {$\Tor^R_i(R/I,M)=0$} for all $i\geq n$.
\end{enumerate}
\end{thm}
\begin{proof} (1) Let $\xx=x_1,\dots,x_e$ be a sequence in $R$ such that $I=(\xx)$ and consider the Koszul complex $\K(\xx)$. By definition we have $\qpd_R(R/I)\leq e-\hsup \K(\xx)=\grade_R(I)$. Set $c=\grade I$ and let $\boldsymbol{a}:=a_1,\dots,a_c$ be a regular sequence contained in $I$. Then we have $\Tor^R_i(R/I,R/(\boldsymbol{a}))=0$ for all $i\gg 0$ while $\Tor^R_c(R/I,R/(\boldsymbol{a}))\cong((\boldsymbol{a}):_RI)/(\boldsymbol{a})\neq 0$. Hence by Corollary \ref{17}, we have $\qpd_R(R/I)\geq \grade I$.

Now we show (2). The spectral sequence \eqref{sptor1} in Lemma \ref{spseq} turns into
\begin{equation}\label{sptor}
\E^2_{p,q}=\Tor^R_p(\H_q(\xx),M)\Longrightarrow \H_{p+q}(\xx,M),
\end{equation}
where $\H_i(\xx,M)=\H_i(\K(\xx)\otimes_RM)$.
Note that $\hinf\K(\xx)=0$, $\hsup\K(\xx)=e-\grade I$ and $\hsup\K(\xx)\otimes_RM=e-\grade(I,M)$.
Hence $\sup\{\hsup\K(\xx)-\hinf\K(\xx),\hsup(\K(\xx)\otimes_RM)-\hinf\K(\xx)-1\}\le e-g$.
Therefore (2) immediately follows by Theorem \ref{main1}(1).
\end{proof}

\begin{rem} In Theorem \ref{main}, if $R$ is local then by using Auslander--Buchsbaum formula we get $\grade I=\depth R-\depth R/I$. This provides a generalization of \cite[Theorem 4.1]{AHS} for every FKH ideal.
\end{rem}

\begin{cor}\label{cm}
Let $(R,\m)$ be a Cohen--Macaulay local ring, and let $I$ be an FKH ideal of $R$.
Then $I$ is a Cohen--Macaulay ideal of $R$.
\end{cor}

\begin{proof}
By Auslander--Buchsbaum formula and Theorem \ref{main}(1) we get $\depth R/I = \depth R - \qpd_RR/I=\dim R-\grade(I)\geq\dim R/I$. This finishes the proof.
\end{proof}

\begin{lem}\label{khom} Let $R$ be a ring and let $I$ be an ideal of $R$. Suppose $I$ is generated with $x_1,\dots,x_n$ such that $\xx= x_1,\dots, x_c$ is a regular sequence for some $0\leq c\leq n$. Then $\H_i(x_1,\dots,x_n)\cong \H_i(\overline{x}_{c+1},\dots,\overline{x}_n)$ where $\overline{x}_i$ is the image of $x_i$ over $R/(\xx)$.
\end{lem}
\begin{proof} It follows from \cite[Theorem 1.6.13(b)]{BH} by setting $M=R$.
\end{proof}

Peskine and Szpiro \cite{PS} introduced theory of linkage of ideals over commutative Noetherian rings. Two nonzero ideals $I$ and $J$ are said to be {\em linked}  if there exists a regular sequence $\xx=x_1,\dots,x_n\in I\cap J$ such that $((\xx):_RI)=J$ and $((\xx):_RJ)=I$. In \cite{PS} they proved that if $R$ is Gorenstein local ring and $I$ is a perfect ideal (i.e $\pd_RR/I=\grade_RI$)  then $J$ also is perfect.

We call an ideal $I$ {\em quasi-perfect} if $\qpd_RR/I=\grade_RI$. The following result partially provides an analogue of  theorem of Peskine and Szpiro for quasi-perfect ideals. In general, we don't know that over a Gorenstein local ring $R$, if $I$ is a quasi-perfect ideal of $R$ then its linked ideal is also quasi-perfect.

\begin{thm}  Let $(R,\m)$ be a  local ring and let $I$ be an FKH ideal of $R$ with finite G-dimension. If $J$ is an ideal linked to $I$, then $J$ is quasi-perfect.
\end{thm}
\begin{proof} Suppose $I$ is linked to $J$. Then by definition, there is a regular sequence $\xx=x_1,\dots,x_g \in I\cap J$ such that $((\xx):_RI)=J$ and $((\xx):_RJ)=I$. Note that since $I$ and $J$ strictly contain $(\xx)$, we have $\grade_R I=\grade_R J=g$.
First, we show that $I/(\xx)$ is an FKH ideal of $R/(\xx)$. By induction, it is enough to show that if $x\in I$ is a regular element, then $I/(x)$ is an FKH ideal of $R/(x)$. Let $\overline{\K(I)}=\K(I)\otimes_RR/(x)$. If $x\in I\backslash \m I$ then Lemma \ref{khom} implies that $I/(x)$ is an FKH ideal over $R/(x)$. Assume $x\in \m I$.
The exact sequence $0 \to \K(I) \xrightarrow{x} \K(I) \to \overline{\K(I)} \to 0$ of complexes gives a long exact sequence of homologies
$$\dots \to \H_i(I) \xrightarrow{x} \H_i(I) \to \H_i(\overline{\K(I)}) \to \H_{i-1}(I) \xrightarrow{x} \H_{i-1}(I) \to \dots.$$
Since $x\H_i(I)=0$ the last long exact sequence breaks into short exact sequences
$0\to \H_i(I) \to \H_i(\overline{\K(I)}) \to \H_{i-1}(I) \to 0$ for all $i\geq 0$. Since $\overline{\K(I)}$ is a Koszul complex over $R/(x)$ with respect to a sequence of minimal generators of $I/(x)$, one has $\H_i(\overline{\K(I)}) $ is an $R/I$-module for all $i$. Therefore the last short exact sequences splits, and so $\H_i(\overline{\K(I)}) $ is a free $R/I$-module for all $i$. Thus $I/(x)$ is an FKH ideal of $R/(x)$.

Set $S=R/(\xx)$. Since  $\Gdim_RR/I<\infty$ and $\xx$ is a regular sequence, we have $\Gdim_SR/I<\infty$.
By Theorem \ref{main}(1), $\qpd_SR/I=\grade_SI/(\xx)=0$. Therefore one has $\Ext^{i>0}_S(R/I,S)=0$ by Corollary \ref{17}, and so $R/I$ is a totally reflexive $S$-module. 
Since $R/J \cong \Omega^{-1}(\Hom_S(R/I,S))$, it follows from  Proposition \ref{mstar} and the Auslander--Buchsbaum formula \ref{AB}  that $\qpd_SR/J = 0$.  Now, Proposition \ref{7}(3) says that $\qpd_RR/J\leq g=\grade_RJ$. This finishes the proof.
\end{proof}

\begin{defn}
A proper ideal $I$ of $R$ is called \emph{quasi-Gorenstein} if
$$
\Ext^i_R(R/I,R)\cong \left\lbrace
           \begin{array}{c l}
              R/I\ \ & \text{ \ \ $i=\grade I$,}\\
              0\ \   & \text{   \ \ $\textrm{otherwise}$.}
           \end{array}
        \right.
$$
\end{defn}

Let $R$ be a local ring. We denote by $\nu(M)$ the minimal number of generators of an $R$-module $M$.

\begin{prop}\label{gdim}
Let $I$ be an FKH ideal of a local ring $R$.
If $\Ext^{\gg 0}_R(I,R)=0$, then $I$ is a quasi-Gorenstein ideal.
Moreover, for all $i\geq 0$ one has $\H_i(I)\cong \H_{n-g-i}(I)$, where $n=\nu(I)$ and $g=\grade I$.
\end{prop}
\begin{proof} By Lemma \ref{spseq} we have a spectral sequence
\begin{equation}\label{spext}
\E^{p,q}_2=\Ext^p_R(\H_q(I),M) \Longrightarrow \H^{p+q}(I,M),
\end{equation}
where $\H^i(I,M)=\H^i(\Hom_R(\K(I),M))$.
Theorem \ref{main}(1) and Corollary \ref{17} shows $\Ext^i_R(R/I,R)=0$ for all $i\neq g$, and the spectral sequence \eqref{spext} implies that $\Ext^g_R(R/I,R)\cong \H^g(I)\cong \H_{n-g}(I)$.
Therefore $\Ext^g_R(R/I,R)$ is a free $R/I$-module.
Thus $\Ext^g_R(R/I,R)\cong (R/I)^m$.
We only need to show $m=1$.
Without loss of generality, we may assume $g=0$; see Lemma \ref{khom}.
Let $J=(0:_RI)$.
Since $\Ext^1_R(R/I,R)=0$, we have $\Hom_R(I,R)\cong R/J$.
Also, $\Hom_R(R/J,R)\cong I$.
It follows that $\Hom_R(\Hom_R(R/J,R),R)\cong R/J$.
Now the commutative diagram
$$
\xymatrix@R-1pc{
0\ar[r]& J\ar[r]\ar[d]& R\ar[r]\ar[d]_\cong& R/J\ar[r]\ar[d]_\cong& 0\\
0\ar[r]& J^{\ast\ast}\ar[r]& R^{\ast\ast}\ar[r]& (R/J)^{\ast\ast}\ar[r]& 0
}
$$
implies that $J\cong J^{**}$. Since $J\cong \Hom_R(R/I,R)\cong(R/I)^m$, the isomorphism $J\cong J^{**}$ says that $m=m^3$ whence $m=1$.
The last part follows from the spectral sequence (\ref{spext}).
\end{proof}

The following result provides an affirmative answer to Question \ref{q2} in a special case.

\begin{cor}\label{sym}
Let $R$ be a local ring, and let $I$ be a FKH ideal of $R$.
Let $M$ be an $R$-module and assume $\Gdim_RI<\infty$.
Then $\Tor^R_{\gg0}(R/I,M)=0$ if and only if $\Ext^{\gg 0}_R(R/I,M)=0$.
\end{cor}

\begin{proof}
Assume $\Tor^R_{\gg0}(R/I,M)=0$.
Consider the exact sequence $0\to \Omega M \to F \to M \to 0$, where $F$ is a free module.
Since $\Gdim I<\infty$, one has $\Ext^{\gg 0}_R(R/I,M)=0$ if and only if $\Ext^{\gg 0}_R(R/I,\Omega M)=0$.
Therefore, taking a higher syzygy, we may assume that $\Tor^R_{i>0}(R/I,M)=0$.
Also, note that by Lemma \ref{khom} we may assume $\grade I=0$.
Then by the spectral sequence \eqref{sptor} we get $\H_i(I,M)\cong \H_i(I)\otimes_R M$.
In particular, $(0:_MI)=\H_n(I,M)\cong (0:_RI)\otimes_RM$, where $n=\nu(I)$.
This shows $\Hom_R(R/I,M)\cong \Hom_R(R/I,R)\otimes_RM$.
Since $\H_i(I)$ is a free $R/I$-module for all $i\geq 0$, we get
$$
\Hom_R(\H_i(I),M)\cong \Hom_R(\H_i(I),R)\otimes_RM\cong \H_{n-i}(I)\otimes_RM\cong \H_{n-i}(I,M) \cong \H^{i}(I,M).
$$
The spectral sequence \eqref{spext} gives $\Hom_R(\H_q(I),M)=\E^{0,q}_2\cong \E^{0,q}_{\infty}$.
It follows that $\Ext^{>0}_R(R/I,M)=0$.
The converse follows by a similar argument.
\end{proof}

\end{document}